\documentclass[12pt,reqno]{amsart}
\usepackage{}
\usepackage{amsmath}
\usepackage{mathrsfs}
\usepackage{amssymb}
\usepackage{amsthm}
\usepackage{mathrsfs}
\usepackage[centertags]{amsmath}
\usepackage{amsfonts}
\usepackage[numbers,sort&compress]{natbib}
\usepackage{color}
\usepackage{extarrows}
\usepackage{fullpage}
\usepackage{CJK}
\usepackage[colorlinks=true,  linkcolor=blue, citecolor=blue]{hyperref}
\input amssym.def
\input amssym.tex

\date{}

\newtheorem{Theorem}{Theorem}[section]

\newtheorem{Proposition}{Proposition}[section]

\newtheorem{Lemma}{Lemma}[section]

\newcommand\R{\mbox{\bf R}}
\newcommand\T{\mbox{\bf T}}

\newcommand\SR{\mbox{\scriptsize\bf R}}

\newcommand{\definition}{{\lower .5ex
  \hbox{$\>\>\stackrel{\triangle}{=}\>\>$} }}
\newcommand\supp{\mathop{\rm supp}}



\allowdisplaybreaks
\begin{document}

\baselineskip=22pt
\thispagestyle{empty}

\begin{center}
{\Large \bf  The Cauchy problem for the generalized KdV equation  in the Sobolev space $H^{s}(\R)$}\\[1ex]

{ Xiangqian Yan\footnote{Email:yanxiangqian213@126.com}$^{a}$,\, Yongsheng Li\footnote{Email:  yshli@scut.edu.cn}$^{a}$,\,Juan Huang\footnote{Email:  hjmath@163.com}$^{b}$,\,Jianhua Huang\footnote{Email:  jhhuang32@nudt.edu.cn}$^{c}$,\,Wei Yan\footnote{Email:  011133@htu.edu.cn}$^{d}$}\\[1ex]

{$^a$School of Mathematics,
 South China University of Technology,}\\
 {Guangzhou, Guangdong 510640, P. R. China}\\[1ex]

{$^b$School of Mathematical Sciences, Sichuan Normal University,}\\
{Chengdu, Sichuan 610066, P. R. China}\\[1ex]

{$^c$College of Sciences, National University of Defense Technology,}\\
{Changsha, Hunan 410073, P. R.  China}\\[1ex]

{$^d$School of Mathematics and Statistics, Henan
Normal University,}\\
{Xinxiang, Henan 453007, P. R.  China}\\[1ex]

\end{center}
\noindent{\bf Abstract.}
In this paper, we are concerned with the Cauchy problem for the generalized KdV equation  with random data and rough data. Firstly,
when $s\in\R$, by using the initial value randomization technique  introduced by Shen et al. (arXiv:2111.11935) and the construction of appropriate auxiliary spaces, we establish the almost sure local well-posedness of the generalized KdV equation in $H^{s}(\R)$,  which improves Theorem 1.3 of Hwang and Kwak (Proc. Amer. Math. Soc. 146(2018), 267-280.) and Theorem 1.5 of Yan et al.(arXiv:2011.07128.). Secondly, by using the well-posedness results proved in Theorem 1.1,  for $f\in H^{s}(\R),\, s\in\R$,  we obtain
\begin{eqnarray*}
&&\mathbb{P}\left(\left\{\omega:\lim_{t\rightarrow0}\|u(t,x)-U(t)f^{\omega}(x)\|_{L_{x}^{\infty}}=0\right\}\right)=1,
\end{eqnarray*}
which improves Theorem 1.6 of Yan et al.(arXiv:2011.07128.).
Thirdly, by using the dyadic decomposition and constructing appropriate function spaces, we establish nonlinear smoothing for the generalized KdV equation with rough data. Furthermore, by using this estimate, when data $f\in H^{s}(\R)\cap\hat{L}^{\infty}(\R)$,\, $s>\frac{1}{2}-\frac{2}{k+1},\, k\geq4$, we obtain
\begin{eqnarray*}
&&\lim_{|x|\rightarrow \infty}u(t,x)=0,\quad t\in[0, T].
\end{eqnarray*}
In particular, for $f(x)\in H^{s}(\R)$,\,$s>\frac{1}{2}-\frac{2}{k+1}$,\, $k\geq4$, we prove
\begin{eqnarray*}
&&\lim_{|x|\rightarrow \infty}(u(t,x)-U(t)f(x))=0.
\end{eqnarray*}
Finally, by using Theorem 1.1, when $f\in H^{s}(\R),\, s\in\R$, we obtain
\begin{eqnarray*}
&&\mathbb{P}\left(\left\{\omega: \forall t\in I_{\omega}, \lim_{|x|\rightarrow \infty}\left(u(t,x)-U(t)f^{\omega}(x)\right)=0\right\}\right)=1.
\end{eqnarray*}

 \bigskip

\medskip

\medskip
\noindent {\bf Corresponding Author:} Juan Huang

\medskip
\noindent {\bf Email Address:} hjmath@163.com

\medskip
\noindent {\bf MSC2020-Mathematics Subject Classification}: 35Q53; 35R60; 42B37

\leftskip 0 true cm \rightskip 0 true cm

\newpage

\baselineskip=20pt

\bigskip
\bigskip
\tableofcontents

\section{Introduction and the main results}
\bigskip

\setcounter{Theorem}{0} \setcounter{Lemma}{0}\setcounter{Definition}{0}\setcounter{Proposition}{1}

\setcounter{section}{1}

\subsection{Research background and current status}
In this paper, we consider the Cauchy problem for the generalized KdV equation
\begin{eqnarray}
&&u_{t}+u_{xxx}+\frac{1}{k+1}\partial_{x}(u^{k+1})=0.\quad k\in \mathbf{N}^{+},\label{1.01}
\end{eqnarray}
in the Sobolev space $H^{s}(\R)(s\in\mathbf{R})$ with random data $u(0,x)=f^{\omega}(x)$, where $f^{\omega}(x)$ denotes the randomization of $f(x)$ as defined in \eqref{1.03}.

The initial data randomization technique plays a crucial role  in studying the probabilistic local and global well-posedness of the evolution equations in the supercritical regime and   constructing invariant measures \cite{B1994CMP,B1994DUKE,B1996, BB2014JEMS,BB2014JFA, BOP,BOP2015,BT2007,BT2008I,BT2008II,BT2014,CO2012,DLM2019,D2015,DC2011,FS2015,HO2015,HO2016,HK2018, LRS1988,LM2014, L2024,NPS2013,NOBS2012,O2009CMP,O2009SIAM,O2009DIE,OQ2013,OP2016,ORT2016,P2017,PRT2014,R2016,S2012,ST2014,SSW2021,T2006,T2009,TT2010,TV2013,TV2014,
YYDH2020, ZF2011,ZF2012}. \eqref{1.01} has been studied by some people. Oh et al. \cite{ORT2016} and Richards \cite{R2016}  studied the invariance of the Gibbs measure for the generalized KdV equation and  the periodic quartic gKdV equation, respectively. For $s>\max\left\{\frac{1}{k+1}\left(\frac{1}{2}-\frac{2}{k}\right),\,\frac{1}{4}-\frac{2}{k}\right\}$, $k\geq5$,  Hwang and Kwak \cite{HK2018}  established the probabilistic local well-posedness of \eqref{1.01} in $H^{s}(\R)$. Recently, Yan et al. \cite{YYDH2020} improved this result to the range $s>\max\left\{\frac{1}{k+1}\left(\frac{1}{2}-\frac{2}{k}\right),\,\frac{1}{6}-\frac{2}{k}\right\}$, $k\geq5$. Compaan et al. \cite{CLS2021} studied the pointwise convergence problem  of the nonlinear Schr\"{o}dinger equation with random data and rough data. Yan et al. \cite{YYDH2020} studied the pointwise convergence problem of the generalized KdV equation with rough data and random data. Very recently, by using frequency-restricted estimates, Correia et al. \cite{COS2024} and Correia and Leite \cite{CL2025} obtained the nonlinear smoothing effects for the quartic KdV equation and the generalized KdV equation \eqref{1.01} with $k\geq4$, respectively.

\subsection{Motivations and contributions}

On the one hand,  for $s\in \R$, when data belongs to $H^{s}(\R^{d})(d=3,4)$,  Shen et al. \cite{SSW2021} investigated the almost sure scattering problem for the defocusing energy-critical nonlinear Schr\"{o}dinger equations.  A natural question arises: by applying the initial data randomization procedure in \cite{SSW2021}, does the probabilistic well-posedness of the generalized KdV equation hold for rough data in $H^{s}(\R),\,s\in\R$? On the other hand, by using frequency-restricted estimates, Correia and Leite  \cite{CL2025} obtained nonlinear smoothing effects for the generalized KdV equation with rough data. A natural question arises: can we find a method different from that in \cite{CL2025} to obtain the same results?

We provide positive answers to both questions raised above. In this paper, we consider the Cauchy problem for the generalized KdV equation  with rough data  and random data. Firstly,
for initial data in $H^{s}(\R)$ with $s\in\R$, we establish the almost sure local well-posedness of \eqref{1.01}. The proof is based on the technique of initial data randomization \cite{SSW2021} and the construction of appropriate auxiliary spaces. This result improves \cite[Theorem 1.3]{HK2018} and \cite[Theorem 1.5]{YYDH2020}.
Secondly, we establish the almost sure convergence of the inhomogeneous part of the  solution to the generalized KdV equation in the integral form with random data, more precisely, for $f\in H^{s}(\R),\, s\in\R,$ we obtain
\begin{eqnarray*}
&&\mathbb{P}\left(\left\{\omega:\lim_{t\rightarrow0}\|u(t,x)-U(t)f^{\omega}(x)\|_{L_{x}^{\infty}}=0\right\}\right)=1,
\end{eqnarray*}
which improves \cite[Theorem 1.6]{YYDH2020}.
 Thirdly, by using a method similar to those in \cite{RV2012,YWY2025}, through dyadic decomposition and constructing appropriate function spaces, we establish nonlinear smoothing for the generalized KdV equation with rough data, more precisely, we obtain
\begin{align*}
&\quad\left\|\int_{0}^{t}U(t-\tau)\partial_{x}\left(\prod_{j=1}^{k+1}u_{j}\right)d\tau\right\|_{X_{T}^{\mu+s+2\epsilon}}\nonumber\\
&\leq
C(T)\left(\sum_{j\geq0}2^{2(\mu+s)j}\left\|\Delta_{j}
\left(\prod_{j=1}^{k+1}u_{j}\right)\right\|_{L_{x}^{1}L_{T}^{\frac{2}{1+\epsilon_{1}}}}^{2}\right)^{\frac{1}{2}}\nonumber\\
&\leq C(T)\prod_{j=1}^{k+1}\|u_{j}\|_{X_{T}^{s}},
\end{align*}
where $C(T)=C\max\left\{T^{\frac{4\epsilon^{2}}{3+2\epsilon}},\, T^{\frac{8\epsilon^{2}}{3+2\epsilon}},\, T^{\frac{4\epsilon^{2}}{3+2\epsilon}+\frac{\epsilon_1}{k}} \right\}$, $k\geq4$ and $s\geq\frac{\mu-2}{k}+\frac{1}{2}+\frac{6\epsilon}{k}+\epsilon_{2}$,\, $\epsilon,\, \epsilon_{2}>0$,\, $\epsilon_{1}=\frac{(1-2\epsilon)4\epsilon}{3+2\epsilon}$,\, $0\leq \mu\leq \min\{1-6\epsilon,\,1-6\epsilon+k(s-(\frac{1}{2}-\frac{1}{k})-\epsilon_{2})\}$.
Fourthly, by applying Theorem 1.3, when rough data $f\in H^{s}(\R)\cap\hat{L}^{\infty}(\R)$,\, $s>\frac{1}{2}-\frac{2}{k+1},\, k\geq4$, we have
\begin{eqnarray*}
&&\lim_{|x|\rightarrow \infty}u(t,x)=0,\quad t\in[0, T].
\end{eqnarray*}
In particular, for $f(x)\in H^{s}(\R)$,\,$s>\frac{1}{2}-\frac{2}{k+1}$,\, $k\geq4$, we have
\begin{eqnarray*}
&&\lim_{|x|\rightarrow \infty}(u(t,x)-U(t)f(x))=0.
\end{eqnarray*}
Finally, we establish the almost sure spatial decay of the inhomogeneous part of the solution in the integral form. More precisely, for $f\in H^{s}(\R),\, s\in\R$, we obtain
\begin{eqnarray*}
&&\mathbb{P}\left(\left\{\omega: \forall t\in I_{\omega}, \lim_{|x|\rightarrow \infty}\left(u(t,x)-U(t)f^{\omega}(x)\right)=0\right\}\right)=1.
\end{eqnarray*}

\subsection{Introduction to notations and spaces}
Before presenting the main results, we introduce the following notations, which will be used in the proving process.  $\rm{mes}(E)$ denotes the Lebesgue measure of $E$, and
$a\sim b$ means that there exists $C_{1}, C_{2}>0$ satisfying $C_{1}|a|\leq|b|\leq C_{2}|a|$. Let $\phi(\xi)$ be a function in $C_{c}^{\infty}(\R)$ with $\supp \phi\subset(\frac{1}{2},2]$. Let $a_{1}>0$, $d\geq1$, $N\in 2^{\mathbf{N}}$. We define $ \langle x\rangle=(1+|x|^{2})^{\frac{1}{2}}$,
\begin{eqnarray*}
&&\hat{f}(\xi)=\frac{1}{(2\pi)^{\frac{d}{2}}}\int_{\SR^{d}}e^{-ix\cdot\xi}f(x)dx,\quad \check{f}(x)=\frac{1}{(2\pi)^{\frac{d}{2}}}\int_{\SR^{d}}e^{ix\cdot\xi}\hat{f}(\xi)d\xi,\\
&&\mathscr{F}_{x}f(\xi)=\frac{1}{\sqrt{2\pi}}\int_{\SR}e^{-ix\xi}f(x)dx,\quad \mathscr{F}_{\xi}^{-1}f(x)=\frac{1}{\sqrt{2\pi}}\int_{\SR}e^{ix\xi}\mathscr{F}_{x}f(\xi)d\xi,\\
&&P_{\leq a_{1}}f=\frac{1}{\sqrt{2\pi}}\int_{|\xi|\leq a_{1}}e^{ix\xi}\mathscr{F}_{x}f(\xi)d\xi,\quad P_{\geq a_{1}}f=\frac{1}{\sqrt{2\pi}}\int_{|\xi|\geq a_{1}}e^{ix\xi}\mathscr{F}_{x}f(\xi)d\xi,\\
&&P_{N}f=\frac{1}{\sqrt{2\pi}}\int_{\SR}e^{ix\xi}\phi\left(\frac{\xi}{N}\right)\mathscr{F}_{x}f(\xi)d\xi,
\quad\|f^{\omega}\|_{L_{\omega}^{p}L_{x}^{q}L_{t}^{r}}=\left(\int_{\Omega}\left(\|f^{\omega}\|_{L_{x}^{q}L_{t}^{r}}\right)^{p}d\mathbb{P}(\omega)\right)^{\frac{1}{p}},\\
&&\|f\|_{L_{x}^{p}L_{t}^{q}(\SR\times I)}=\left(\int_{\SR}\left(\int_{I}|f|^{q}dt\right)^{\frac{p}{q}}dx\right)^{\frac{1}{p}},\quad U(t)f=\frac{1}{\sqrt{2\pi}}\int_{\SR}e^{ix\xi}e^{it\xi^{3}}\mathscr{F}_{x}f(\xi)d\xi.
\end{eqnarray*}
The space $X^{s,b}$ is defined as follows
\begin{eqnarray*}
&&X^{s,b}=\{u\in\mathcal{S}^{\prime}(\R^{2}):\|u\|_{X^{s,b}}=\|\langle\tau-\xi^{3}\rangle^{b}\langle\xi\rangle^{s}\mathscr{F}u(\tau,\xi)\|_{L_{\tau\xi}^{2}}<\infty\}.
\end{eqnarray*}

\subsection{Initial value randomization}

Inspired by \cite{SSW2021}, we consider the following randomization process. Suppose that $N\in 2^{\mathbf{N}}$, we define
\begin{eqnarray*}
&&O_{N}=\{\xi\in\R:|\xi|\leq N\},\quad Q_{N}=O_{2N}\setminus O_{N}.
\end{eqnarray*}
and for $a\in\mathbf{N}$, we define
\begin{eqnarray*}
&&\mathcal{A}(Q_{N})=\{Q: Q\,\, \text{is a dyadic interval with length}\,\, N^{-a},\,Q\subset Q_{N}\},
\end{eqnarray*}
where $\sharp \mathcal{A}(Q_{N})=N^{a+1}$ and $\bigcup_{Q\in\mathcal{A}(Q_{N})}Q=Q_{N}$.
We denote
\begin{eqnarray*}
&&\mathcal{Q}=O_{1}\cup\{Q: Q\in\mathcal{A}(Q_{N}),\,N\in 2^{\mathbf{N}}\}.
\end{eqnarray*}
By rearranging the elements in $\mathcal{Q}$, we can obtain
\begin{eqnarray*}
&&\mathcal{Q}=\{Q_{k}:k\in\mathbf{N}\},\quad \R=\bigcup_{Q\in\mathcal{Q}}Q=\bigcup_{k\in\mathbf{N}}Q_{k}.
\end{eqnarray*}
We choose $\psi_{k}(\xi)\in C_{0}^{\infty}(\R)$ to be a real-valued function, satisfying $\psi_{k}(\xi)=1$,\, $\xi\in Q_{k}$ and $\psi_{k}(\xi)=0$,\, $\xi\notin2Q_{k}$, and for all $\xi\in \R$,
\begin{eqnarray*}
&&\sum_{k\in\mathbf{N}}\psi_{k}(\xi)=1,
\end{eqnarray*}
and $\psi_{k}(D)$ is the pseudodifferential operator defined by
\begin{eqnarray*}
&&(\psi_{k}(D)f)(x)=\mathscr{F}_{\xi}^{-1}\big(\psi_{k}(\xi)\mathscr{F}_{x}f\big)(x),x\in \R.
\end{eqnarray*}
Let $\{g_{k}(\omega)\}_{k\in \mathbf{N}}$ be a sequence of independent real-valued random variables with zero mean on the probability space $(\Omega,\mathcal{A}, \mathbb{P})$. Their probability distributions are denoted by $\mu_{k}(k\in\mathbf{N})$, and this family satisfies the following properties:
for $\forall \gamma_{k}\in \R,\, \forall k\in\mathbf{N}$,\, $\exists c>0$
\begin{eqnarray}
&&\Big|\int_{-\infty}^{+\infty}e^{\gamma_{k} x}d\mu_{k}(x)\Big|\leq e^{c\gamma_{k}^2}.\label{1.02}
\end{eqnarray}
Then, we define the randomization of $f$ as follows
\begin{eqnarray}
&&f^{\omega}(x)=\sum_{k\in\mathbf{N}}g_{k}(\omega)\psi_{k}(D)f.\label{1.03}
\end{eqnarray}

\subsection{The main results }

\begin{Theorem}\label{Theorem1}(Almost sure local well-posedness.)
 Suppose that $s\in\R$,\,  $k\geq5$,\, $0<s_{1}<\frac{1}{9}$,\, $\sigma>\frac{1}{2}-\frac{2}{k}+\frac{4s_{1}}{k}$. Let  $f(x)\in H^{s}(\R)$.\, Its randomization $f^{\omega}(x)$ is defined as in \eqref{1.03}. Take $a\in \mathbf{N}$ satisfying
\begin{eqnarray*}
&a>\max\left\{\frac{2(k+1)(\sigma+2s_{1}-s)-2s_{1}}{k-4+4s_{1}},\, -\frac{2ks}{k-4-4s_{1}},\,\frac{\sigma-s}{4s_{1}}+\frac{1}{4},\,
-\frac{2(ks-2s_{1})}{k-4+4s_{1}} \right\}.
\end{eqnarray*}
Then, we have that \eqref{1.01} is almost sure local well-posedness with the random data $f^{\omega}(x)$.  More precisely, for almost every $\omega\in\Omega$, there exists a time interval $I_{\omega}$ containing $0$, \eqref{1.01} admits a unique solution $u$ with the random data $f^{\omega}(x)$, satisfying
\begin{eqnarray*}
&&u-U(t)f^{\omega}\in C(I_{\omega}, H^{\sigma}(\R)).
\end{eqnarray*}

\end{Theorem}

\noindent{\bf Remark 1.} In \cite[Theorem 1.3]{HK2018}, they require $s>\max\left\{\frac{1}{k+1}\left(\frac{1}{2}-\frac{2}{k}\right),\,\frac{1}{4}-\frac{2}{k}\right\}$,\,
$k\geq 5$ and in \cite[Theorem 1.5]{YYDH2020}, they require $s>\max\left\{\frac{1}{k+1}\left(\frac{1}{2}-\frac{2}{k}\right),\,\frac{1}{6}-\frac{2}{k}\right\}$,\,
$k\geq 5$.
In our results, we remove the requirements on $s$.

\begin{Theorem}\label{Theorem2}(Almost sure convergence of the inhomogeneous part of the solution in the integral form.)
Suppose that $s\in\R$,\,  $k\geq5$,\, $0<s_{1}<\frac{1}{9}$,\, $\epsilon>0$,\, $\sigma=\frac{1}{2}+\epsilon$. Let  $f(x)\in H^{s}(\R)$.\,  $f^{\omega}(x)$ is defined as in \eqref{1.03}. Take $a\in \mathbf{N}$ satisfying
\begin{eqnarray*}
&a>\max\left\{\frac{2(k+1)(\sigma+2s_{1}-s)-2s_{1}}{k-4+4s_{1}},\, -\frac{2ks}{k-4-4s_{1}},\,\frac{\sigma-s}{4s_{1}}+\frac{1}{4},\,
-\frac{2(ks-2s_{1})}{k-4+4s_{1}} \right\}.
\end{eqnarray*}
and  let $u(t,x)$ be the solution to  \eqref{1.01}. Then we have
\begin{eqnarray}
&&\mathbb{P}\left(\left\{\omega:\lim_{t\rightarrow0}\|u(t,x)-U(t)f^{\omega}(x)\|_{L_{x}^{\infty}}=0\right\}\right)=1. \label{1.04}
\end{eqnarray}

\end{Theorem}

\noindent{\bf Remark 2.} In \cite[Theorem 1.6]{YYDH2020}, they require $s>\frac{1}{6}$,\,
$k\geq 6$. In our results, we remove the requirements on $s$ and relax $k$ to $k\geq5$.

In order to prove Theorem 1.3, we introduce the following dyadic mixed Lebesgue spaces
\begin{eqnarray*}
&&\|u\|_{X_{T}^{s}}=\|\|\Delta_{j}u\|_{Y_{T}^{s}}\|_{\ell^{2}},\\ &&\|u\|_{Y_{T}^{s}}=\|u\|_{L_{T}^{\infty}H^{s}}+\|\langle D\rangle^{s+1-2\epsilon}u\|_{L_{x}^{\infty}L_{T}^{\frac{2}{1-\epsilon_{1}}}}
+\|\langle D\rangle^{s-\frac{k-2}{2k}-\epsilon_{2}}u\|_{L_{x}^{k}L_{T}^{\frac{k}{\epsilon_{1}}}},
\end{eqnarray*}
where $\epsilon_{1}=\frac{(1-2\epsilon)4\epsilon}{3+2\epsilon}$ and $(\Delta_{j})_{j\geq0}$ is a Littlewood-Paley analysis such that $\Delta_{0}$ localizes spatial frequencies in a ball $|\xi|\leq1 $ and $\Delta_{j} (j >0)$ in an annulus $|\xi|\sim 2^{j}$.

\begin{Theorem}\label{Theorem3}(Nonlinear smoothing.)
Suppose that $k\geq4$ and $s\geq\frac{\mu-2}{k}+\frac{1}{2}+\frac{6\epsilon}{k}+\epsilon_{2}$,\, $\epsilon,\, \epsilon_{2}>0$,\, $\epsilon_{1}=\frac{(1-2\epsilon)4\epsilon}{3+2\epsilon}$,\, $0\leq \mu\leq \min\{1-6\epsilon,\,1-6\epsilon+k(s-(\frac{1}{2}-\frac{1}{k})-\epsilon_{2})\}$.
Then,  we have
\begin{align}
&\quad\left\|\int_{0}^{t}U(t-\tau)\partial_{x}\left(\prod_{j=1}^{k+1}u_{j}\right)d\tau\right\|_{X_{T}^{\mu+s}}\nonumber\\
&\leq
C(T)\left(\sum_{j\geq0}2^{2(\mu+s+2\epsilon)j}\left\|\Delta_{j}
\left(\prod_{j=1}^{k+1}u_{j}\right)\right\|_{L_{x}^{1}L_{T}^{\frac{2}{1+\epsilon_{1}}}}^{2}\right)^{\frac{1}{2}}\nonumber\\
&\leq C(T)\prod_{j=1}^{k+1}\|u_{j}\|_{X_{T}^{s}},\label{1.05}
\end{align}
where $C(T)=C\max\left\{T^{\frac{4\epsilon^{2}}{3+2\epsilon}},\, T^{\frac{8\epsilon^{2}}{3+2\epsilon}},\, T^{\frac{4\epsilon^{2}}{3+2\epsilon}+\frac{\epsilon_1}{k}} \right\}$.

\end{Theorem}

\noindent{\bf Remark 3.} Correia and Leite \cite{CL2025} obtained the same nonlinear smoothing effect using frequency-restricted estimates, but our method of proof is entirely different from that in \cite[Theorem 1.1]{CL2025}. When $\mu+s=\frac{1}{2}+\epsilon$, we obtain $s>\frac{1}{2}-\frac{2}{k+1}$, which is the same as \cite[Theorem 1.6]{YYDH2020}, but our method of proof is entirely different from that in \cite[Theorem 1.6]{YYDH2020}.
Following the ideas in \cite{RV2012,YWY2025}, we prove Theorem 1.3.

\begin{Theorem}\label{Theorem4}(Spatial decay.)
Suppose that $k\geq4$,\, $s>\frac{1}{2}-\frac{2}{k+1}$,\, $f(x)\in H^{s}(\R)\cap\hat{L}^{\infty}(\R)$ and $u(t,x)$ is the solution of \eqref{1.01} with $u(0,x)=f(x)$.
Then, there exists a sufficiently small $T>0$  such that for $t\in[0,T]$, we have
\begin{eqnarray}
&&\lim_{|x|\rightarrow \infty}u(t,x)=0.\label{1.06}
\end{eqnarray}
In particular, for $f(x)\in H^{s}(\R)$,\,$s>\frac{1}{2}-\frac{2}{k+1}$,\, $k\geq4$, we have
\begin{eqnarray}
&&\lim_{|x|\rightarrow \infty}(u(t,x)-U(t)f(x))=0.\label{1.07}
\end{eqnarray}

\end{Theorem}

\begin{Theorem}\label{Theorem5}(Almost sure spatial decay of the inhomogeneous part of the solution in the integral form.)
Assume the conditions of Theorem 1.2 hold and let $u(t,x)$ be the solution to equation \eqref{1.01}. Then, we have
\begin{eqnarray}
&&\mathbb{P}\left(\left\{\omega: \forall t\in I_{\omega}, \lim_{|x|\rightarrow \infty}\left(u(t,x)-U(t)f^{\omega}(x)\right)=0\right\}\right)=1,\label{1.08}
\end{eqnarray}
where $I_{\omega}$ appears in Theorem 1.1.

\end{Theorem}
\noindent{\bf Remark 4.} In \eqref{1.07} of Theorem 1.4, we require $s>\frac{1}{2}-\frac{2}{k+1}$,\, $k\geq4$. By using the technique of initial data randomization \cite{SSW2021}, we remove the requirements on $s$ in Theorem 1.5.

The remainder of this paper is organized as follows. In Section 2,  we give some
preliminaries. In Section 3,
 we prove Proposition 3.1 and Lemma 3.1, which are key to proving Theorem 1.1. In Section 4, we prove
 Theorem 1.1. In Section 5, we prove
 Theorem 1.2. In Section 6, we prove
 Theorem 1.3. In Section 7, we prove
 Theorem 1.4. In Section 8, we prove
 Theorem 1.5.

\bigskip
\bigskip

\section{Preliminary lemmas}

\setcounter{equation}{0}

\setcounter{Theorem}{0}

\setcounter{Lemma}{0}

\setcounter{Proposition}{0}

\setcounter{section}{2}
In this section, we present some lemmas, which will be used in the subsequent proofs.

\begin{Lemma}\label{lem2.1}
Suppose that $k\geq4$,\, $4\leq q<\infty$,\, $p>4$,\, $2\leq r<\infty$,\, $\epsilon,\, \epsilon_{2}>0$,\, $\epsilon_{1}=\frac{(1-2\epsilon)4\epsilon}{3+2\epsilon}$ and $\frac{2}{q}+\frac{1}{r}=\frac{1}{2}$,\, $\alpha=\frac{2}{r}-\frac{1}{q}$,\, $s>\frac{1}{2}-\frac{1}{p}$. Then, for $f\in L^{2}(\R)$, we have
\begin{eqnarray}
&&\left\|DU(t)f\right\|_{L_{x}^{\infty}L_{t}^{2}}\leq C\|f\|_{L^{2}},\label{2.01}\\
&&\left\|D^{-\frac{1}{4}}U(t)f\right\|_{L_{x}^{4}L_{t}^{\infty}}\leq C\|f\|_{L^{2}},\label{2.02}\\
&&\left\|D^{\alpha}U(t)f\right\|_{L_{x}^{q}L_{t}^{r}}\leq C\|f\|_{L^{2}},\label{2.03}\\
&&\left\|\langle D\rangle U(t)f\right\|_{L_{x}^{\infty}L_{t}^{2}}\leq C\|f\|_{L^{2}},\label{2.04}\\
&&\left\|\langle D\rangle^{-\frac{1}{4}}U(t)f\right\|_{L_{x}^{4}L_{t}^{\infty}}\leq C\|f\|_{L^{2}},\label{2.05}\\
&&\left\|\langle D\rangle^{-s}U(t)f\right\|_{L_{x}^{p}L_{t}^{\infty}}\leq C\|f\|_{L^{2}},\label{2.06}\\
&&\left\|\langle D\rangle^{-\frac{k-2}{2k}-\epsilon_{2}}U(t)f\right\|_{L_{x}^{k}L_{T}^{\frac{k}{\epsilon_{1}}}}\leq CT^{\frac{\epsilon_{1}}{k}}\|f\|_{L^{2}}.\label{2.07}
\end{eqnarray}
where
\begin{eqnarray}
&&U(t)f=\int_{\SR}e^{ix\xi}e^{it\xi^{3}}\mathscr{F}_{x}f(\xi)d\xi.\label{2.08}
\end{eqnarray}

\end{Lemma}
\noindent{\bf Proof.}
For \eqref{2.01} and \eqref{2.02}, we refer to \cite[Theorem 3.5]{KPV1993} and \cite[Lemma 3.29]{KPV1993}, respectively. By applying the interpolation theorem in mixed Lebesgue spaces \cite[Theorem 1]{BP1961} to \eqref{2.01} and \eqref{2.02}, we obtain \eqref{2.03}. For \eqref{2.04}, by using Minkowski inequality, \eqref{2.01} and $H^{\frac{1}{2}+\epsilon}(\R)\hookrightarrow L^{\infty}(\R)$, we have
\begin{align}
\left\|\langle D\rangle U(t)P_{\leq 1}f\right\|_{L_{x}^{\infty}L_{t}^{2}}&\leq C\left\|\langle D\rangle U(t)P_{\leq1}f\right\|_{L_{t}^{2}L_{x}^{\infty}}\nonumber\\
&\leq C\left\|\langle D\rangle^{\frac{3}{2}+\epsilon} U(t)P_{\leq1}f\right\|_{L_{xt}^{2}}
\leq C\|f\|_{L^{2}},\label{2.09}\\
\left\|\langle D\rangle U(t)P_{\geq1}f\right\|_{L_{x}^{\infty}L_{t}^{2}}&\leq C\left\||D| U(t)P_{\geq1}\langle D\rangle |D|^{-1}f\right\|_{L_{x}^{\infty}L_{t}^{2}}\nonumber\\
&\leq C\left\|P_{\geq1}\langle D\rangle |D|^{-1}f\right\|_{L_{x}^{2}}\leq C\|f\|_{L^{2}}.\label{2.010}
\end{align}
Combining \eqref{2.09} with \eqref{2.010}, we have that \eqref{2.04} is valid. From \eqref{2.02}, we obtain \eqref{2.05}. For \eqref{2.06}, noting that
\begin{eqnarray}
&&\|U(t)f\|_{L_{xt}^{\infty}}\leq C\|f\|_{H^{s_{1}}},\quad s_{1}>\frac{1}{2}.\label{2.011}
\end{eqnarray}
By applying the interpolation theorem in mixed Lebesgue spaces \cite[Theorem 1]{BP1961} to \eqref{2.05} and \eqref{2.011}, we obtain \eqref{2.06}. For \eqref{2.07}, by using H\"{o}lder inequality with respect to $t$, and using \eqref{2.05}-\eqref{2.06}, we have
\begin{align}
\left\|\langle D\rangle^{-\frac{k-2}{2k}-\epsilon_{2}}U(t)f\right\|_{L_{x}^{k}L_{T}^{\frac{k}{\epsilon_{1}}}}&\leq CT^{\frac{\epsilon_{1}}{k}}\left\|\langle D\rangle^{-\frac{k-2}{2k}-\epsilon_{2}}U(t)f\right\|_{L_{x}^{k}L_{T}^{\infty}}\nonumber\\
&\leq CT^{\frac{\epsilon_{1}}{k}}\|f\|_{L^{2}}.\label{2.012}
\end{align}
It follows from \eqref{2.012} that \eqref{2.07} holds.

The proof of Lemma 2.1 is completed.

\begin{Lemma}\label{lem2.2}
Suppose that $k\geq4$,\, $T\leq 1$,\, $s,\,t\in I\subset[0,T]$,\, $\epsilon,\, \epsilon_{2}>0$,\, $\epsilon_{1}=\frac{(1-2\epsilon)4\epsilon}{3+2\epsilon}$. Then, we have
\begin{eqnarray}
&&\left\|\int_{0}^{t}U(t-s)f(s)ds\right\|_{L_{x}^{2}}\leq CT^{\frac{4\epsilon^{2}}{3+2\epsilon}}\|\langle D\rangle^{-(1-2\epsilon)}f\|_{L_{x}^{1}L_{T}^{\frac{2}{1+\epsilon_{1}}}},\label{2.013}\\
&&\left\|\int_{I} U(t-s)f(s)ds\right\|_{L_{x}^{2}}\leq CT^{\frac{4\epsilon^{2}}{3+2\epsilon}}\|\langle D\rangle^{-(1-2\epsilon)}f\|_{L_{x}^{1}L_{T}^{\frac{2}{1+\epsilon_{1}}}},\label{2.014}\\
&&\left\|\langle D\rangle^{1-2\epsilon}\int_{0}^{t}U(t-s)f(s)ds\right\|_{L_{x}^{\infty}L_{T}^{\frac{2}{1-\epsilon_{1}}}}\leq CT^{\frac{8\epsilon^{2}}{3+2\epsilon}}\|\langle D\rangle^{-(1-2\epsilon)}f\|_{L_{x}^{1}L_{T}^{\frac{2}{1+\epsilon_{1}}}},\label{2.015}\\
&&\left\|\langle D\rangle^{-\frac{k-2}{2k}-\epsilon_{2}}\int_{0}^{t}U(t-s)f(s)ds\right\|_{L_{x}^{k}L_{T}^{\frac{k}{\epsilon_{1}}}}\leq CT^{\frac{4\epsilon^{2}}{3+2\epsilon}+\frac{\epsilon_{1}}{k}}\|\langle D\rangle^{-(1-2\epsilon)}f\|_{L_{x}^{1}L_{T}^{\frac{2}{1+\epsilon_{1}}}}.\quad\label{2.016}
\end{eqnarray}

\end{Lemma}

\noindent{\bf Proof.}
By using  dual idea, to prove \eqref{2.014}, we only need to prove
\begin{eqnarray}
&&\left|\left\langle\int_{I}U(-s)f(s)ds,\, g\right\rangle\right|\leq CT^{\frac{4\epsilon^{2}}{3+2\epsilon}}\|g\|_{L_{x}^{2}}\|\langle D\rangle^{-(1-2\epsilon)}f\|_{L_{x}^{1}L_{T}^{\frac{2}{1+\epsilon_{1}}}}.\label{2.017}
\end{eqnarray}
By using Minkowski inequality, H\"{o}lder inequality and $H^{\frac{1}{2}+\epsilon}(\R)\hookrightarrow L^{\infty}(\R)$, we have
\begin{align}
\|\langle D\rangle^{-\frac{1}{2}-\epsilon}U(t)f\|_{L_{x}^{\infty}L_{T}^{\frac{1}{\epsilon}}}&\leq C\|\langle D\rangle^{-\frac{1}{2}-\epsilon}U(t)f\|_{L_{T}^{\frac{1}{\epsilon}}L_{x}^{\infty}}\nonumber\\
&\leq CT^{\epsilon}\|U(t)f\|_{L_{T}^{\infty}L_{x}^{2}}\leq CT^{\epsilon}\|f\|_{L^{2}}.\label{2.018}
\end{align}
Applying the interpolation theorem in mixed Lebesgue spaces \cite[Theorem 1]{BP1961} to \eqref{2.04} and \eqref{2.018}, we obtain
\begin{eqnarray}
&&\|\langle D\rangle^{1-2\epsilon}U(t)f\|_{L_{x}^{\infty}L_{T}^{\frac{2}{1-\epsilon_{1}}}}\leq CT^{\frac{4\epsilon^{2}}{3+2\epsilon}}\|f\|_{L^{2}}.\label{2.019}
\end{eqnarray}
Applying H\"{o}lder inequality and \eqref{2.019}, we obtain the following estimate
\begin{align}
\left|\left\langle\int_{I}U(-s)f(s)ds,\, g\right\rangle\right|&=\left|\int_{\SR}\int_{I}U(-s)f(s)ds\overline{g}dx\right|\nonumber\\
&=\left|\int_{\SR}\int_{I}\langle D\rangle^{-(1-2\epsilon)}f(s)\overline{\langle D\rangle^{1-2\epsilon} U(s)g}ds dx\right|\nonumber\\
&\leq \|\langle D\rangle^{1-2\epsilon} U(s)g\|_{L_{x}^{\infty}L_{T}^{\frac{2}{1-\epsilon_{1}}}}\|\langle D\rangle^{-(1-2\epsilon)} f\|_{L_{x}^{1}L_{T}^{\frac{2}{1+\epsilon_{1}}}}\nonumber\\
&\leq CT^{\frac{4\epsilon^{2}}{3+2\epsilon}}\|g\|_{L_{x}^{2}}\|\langle D\rangle^{-(1-2\epsilon)}f\|_{L_{x}^{1}L_{T}^{\frac{2}{1+\epsilon_{1}}}}.\label{2.020}
\end{align}
This proves \eqref{2.014}.

\noindent By using Christ-Kiselev lemma \cite[Theorem 1.2]{CK2001} and \eqref{2.014}, we have that \eqref{2.013} is valid.

\noindent In a similar way, to prove \eqref{2.015}, we only need to prove
\begin{eqnarray}
&&\left\|\langle D\rangle^{(1-2\epsilon)} U(t)\int_{I} U(-s)f(s)ds\right\|_{L_{x}^{\infty}L_{T}^{\frac{2}{1-\epsilon_{1}}}}\leq CT^{\frac{8\epsilon^{2}}{3+2\epsilon}}\|\langle D\rangle^{-(1-2\epsilon)}f\|_{L_{x}^{1}L_{T}^{\frac{2}{1+\epsilon_{1}}}}.\label{2.021}
\end{eqnarray}
From \eqref{2.019} and \eqref{2.014}, we obtain
\begin{align}
\text{The left hand side of  \eqref{2.021}}&\leq CT^{\frac{4\epsilon^{2}}{3+2\epsilon}}\left\|\int_{I} U(-s)f(s)ds\right\|_{L_{x}^{2}}\nonumber\\
&\leq CT^{\frac{8\epsilon^{2}}{3+2\epsilon}}\|\langle D\rangle^{-(1-2\epsilon)}f\|_{L_{x}^{1}L_{T}^{\frac{2}{1+\epsilon_{1}}}}.\label{2.022}
\end{align}
This yields \eqref{2.015}.

\noindent Similarly, to prove \eqref{2.016}, we only need to prove
\begin{eqnarray}
\left\|\langle D\rangle^{-\frac{k-2}{2k}-\epsilon_{2}} U(t)\int_{I} U(-s)f(s)ds\right\|_{L_{x}^{k}L_{T}^{\frac{k}{\epsilon_{1}}}}\leq CT^{\frac{4\epsilon^{2}}{3+2\epsilon}+\frac{\epsilon_{1}}{k}}\|\langle D\rangle^{-(1-2\epsilon)}f\|_{L_{x}^{1}L_{T}^{\frac{2}{1+\epsilon_{1}}}}.\label{2.023}
\end{eqnarray}
Combining \eqref{2.07} and \eqref{2.014} gives
\begin{align}
\text{The left hand side of  \eqref{2.023}}&\leq CT^{\frac{\epsilon_{1}}{k}}\left\|\int_{I} U(-s)f(s)ds\right\|_{L_{x}^{2}}\nonumber\\
&\leq CT^{\frac{4\epsilon^{2}}{3+2\epsilon}+\frac{\epsilon_{1}}{k}}\|\langle D\rangle^{-(1-2\epsilon)}f\|_{L_{x}^{1}L_{T}^{\frac{2}{1+\epsilon_{1}}}},\label{2.024}
\end{align}
which establishes \eqref{2.016}.

The proof of Lemma 2.2 is completed.

\begin{Lemma}\label{lem2.3}
Suppose that $a,\,k\in \mathbf{N}$,\, $4\leq p\leq\infty$. Then, we have
\begin{eqnarray}
&&\left\|D_{x}^{\frac{a}{2}(1-\frac{4}{p})-\frac{1}{p}}U(t)\psi_{k}(D)f\right\|_{L_{x}^{p}L_{t}^{\infty}}\leq C\|\psi_{k}(D)f\|_{L^{2}},\label{2.025}
\end{eqnarray}
where $\psi_{k}(D)f=\mathscr{F}_{\xi}^{-1}(\psi_{k}(\xi)\mathscr{F}_{x}f(\xi))$.
\end{Lemma}

\noindent{\bf Proof.}
From \eqref{2.02}, we have
\begin{eqnarray}
&&\left\|D_{x}^{-\frac{1}{4}}U(t)f\right\|_{L_{x}^{4}L_{t}^{\infty}}\leq C\|f\|_{L^{2}}.\label{2.026}
\end{eqnarray}
On the one hand, by using \eqref{2.026}, we have
\begin{eqnarray}
&&\left\|D_{x}^{-\frac{1}{4}}U(t)\psi_{k}(D)f\right\|_{L_{x}^{4}L_{t}^{\infty}}\leq C\|\psi_{k}(D)f\|_{L^{2}}.\label{2.027}
\end{eqnarray}
On the other hand, by using Hausdorff-Young's inequality, H\"older inequality and the definition of $\mathcal{Q}$, notice that ${\rm{mes}}(\supp \psi_{k})\sim N^{-a}$ and for $\xi\in \supp\psi_{k}$,  we have
$\langle|\xi|\rangle\sim N$ and
\begin{align}
\left\|U(t)\psi_{k}(D)f\right\|_{L_{t}^{\infty}L_{x}^{\infty}}&\leq C\left\|e^{it\xi^{3}}\psi_{k}(\xi)\mathscr{F}_{x}f(\xi)\right\|_{L_{t}^{\infty}L_{\xi}^{1}}\nonumber\\
&\leq ({\rm{mes}}(\supp \psi_{k}))^{\frac{1}{2}}\|\psi_{k}(\xi)\mathscr{F}_{x}f(\xi)\|_{L_{\xi}^{2}}\leq CN^{-\frac{a}{2}}\|\psi_{k}(D)f\|_{L^{2}}\nonumber\\
&\leq C\left\|\langle D\rangle^{-\frac{a}{2}}\psi_{k}(D)f\right\|_{L^{2}}\leq C\left\|D^{-\frac{a}{2}}\psi_{k}(D)f\right\|_{L^{2}}.\label{2.028}
\end{align}
By applying the interpolation theorem in mixed Lebesgue spaces \cite[Theorem 1]{BP1961} to \eqref{2.027} and \eqref{2.028}, we have
\begin{eqnarray}
&&\|U(t)\psi_{k}(D)f\|_{L_{x}^{p}L_{t}^{\infty}}\leq C\|D_{x}^{s}\psi_{k}(D)f\|_{L^{2}},\label{2.029}
\end{eqnarray}
where
\begin{eqnarray}
&&\frac{1}{p}=\frac{\theta}{4}+\frac{1-\theta}{\infty},\, s=\frac{\theta}{4}-\frac{a(1-\theta)}{2}=\frac{1}{p}-\frac{a}{2}(1-\frac{4}{p}).\label{2.030}
\end{eqnarray}

From \eqref{2.029} and \eqref{2.030}, we obtain \eqref{2.025}.

The proof of Lemma 2.3 is completed.

\begin{Lemma}\label{lem2.4}
Suppose that $\{g_ {k}(\omega)\}_{k\in \mathbf{N}}$ is independent, zero mean, real valued random variable sequence with probability distribution $\mu_{k}(k\in\mathbf{N})$  on probability space $(\Omega,\mathcal{A}, \mathbb{P})$, where $\mu_{k}(k\in\mathbf{N})$ satisfy \eqref{1.02}. Then,
for $r\in[2, \infty)$ and $(a_k)_k\in\ell^2$, we have
\begin{eqnarray}
&&\left\|\sum_{k\in\mathbf{N}}g_{k}(\omega)a_{k}\right\|_{L_{\omega}^{r}}\leq C\sqrt{r}\|a_{k}\|_{\ell^{2}}.\label{2.031}
\end{eqnarray}
\end{Lemma}

For Lemma 2.4,  we refer to \cite[Lemma 3.1]{BT2008I}.

\begin{Lemma}\label{lem2.5}
Suppose that $g$ is a real-valued measurable function on a probability space $(\Omega,\mathcal{A}, \mathbb{P})$ and there exist constants $C > 0,\, M> 0$ and $p\geq1$ such that for any $p_{1}\geq p$, the following inequality holds:
\begin{eqnarray*}
&&\|g\|_{L_{\omega}^{p_{1}}}\leq C\sqrt{p_{1}}M.
\end{eqnarray*}
Then, for any $\lambda> 0$, we have
\begin{eqnarray*}
&&\mathbb{P}(\{\omega\in \Omega : |g(\omega)| > \lambda\}) \leq C_{1}\exp\left(-c\frac{\lambda^{2}}{M^{2}}\right),
\end{eqnarray*}
where $c > 0$ and $C_{1}>0$ depends only on $C$ and $p$.
\end{Lemma}

For Lemma 2.5,  we refer to \cite[Lemma 2.9]{SSW2021}.

\begin{Lemma}\label{lem2.6}
Let $f\in H^{s}$. Its randomization $f^{\omega}(x)$ is defined by \eqref{1.03}. For any $\lambda>0$, there exist constants $C_{1}, c>0$,  such that
\begin{eqnarray*}
&&\mathbb{P}(\{\omega\in \Omega : \|f^{\omega}\|_{H^{s}} > \lambda\}) \leq C_{1}\exp\left(-c\frac{\lambda^{2}}{\|f\|_{H^{s}}^{2}}\right).
\end{eqnarray*}
\end{Lemma}

By using \cite[Lemma 2.6]{SSW2021} and a proof similar to \cite[Lemma 2.2]{BOP},  we have that Lemma 2.6 holds.

\begin{Proposition}\label{pro2.1}
Suppose that $a\in\mathbf{N}$,\, $4<p\leq q<\infty$,\, $2\leq r<\infty$,\, $p_{1}\geq\max\{q,\, r\}$ and $p=q(1-\frac{2}{r})$,\, $s_{0}=\frac{2}{r}-\frac{1}{q}+\frac{a}{2q}(q(1-\frac{2}{r})-4)$,\,
$f^{\omega}(x)$ is the randomization of $f(x)$ and is defined by \eqref{1.03}. Then, for $f\in L^{2}(\R)$, we have
\begin{eqnarray}
&&\left\|D^{s_{0}}U(t)f^{\omega}\right\|_{L_{\omega}^{p_{1}}L_{x}^{q}L_{t}^{r}}\leq C\sqrt{p_{1}}\|f\|_{L^{2}}.\label{2.032}
\end{eqnarray}
Moreover, for any $\lambda> 0$, we have
\begin{eqnarray}
&&\mathbb{P}(\{\omega\in \Omega : \left\|D^{s_{0}}U(t)f^{\omega}\right\|_{L_{x}^{q}L_{t}^{r}} > \lambda\}) \leq C_{1}\exp\left(-c\frac{\lambda^{2}}{\|f\|_{L^{2}}^{2}}\right).\label{2.033}
\end{eqnarray}
where $c > 0$ and $C_{1}>0$ depends only on $C$ and $p_{1}$.

\end{Proposition}

\noindent{\bf Proof.}
Notice that $p_{1}\geq\max\{q,\, r\}\geq2$, by using Minkowski inequality and Lemma 2.4, we have
\begin{align}
\left\|D^{s_{0}}U(t)f^{\omega}\right\|_{L_{\omega}^{p_{1}}L_{x}^{q}L_{t}^{r}}&\leq \left\|\|D^{s_{0}}U(t)f^{\omega}\|_{L_{\omega}^{p_{1}}}\right\|_{L_{x}^{q}L_{t}^{r}}\nonumber\\
&\leq C\sqrt{p_{1}}\left\|\|D^{s_{0}}U(t)\psi_{k}(D)f\|_{\ell^{2}}\right\|_{L_{x}^{q}L_{t}^{r}}\nonumber\\
&\leq C\sqrt{p_{1}}\left\|\|D^{s_{0}}U(t)\psi_{k}(D)f\|_{L_{x}^{q}L_{t}^{r}}\right\|_{\ell^{2}}.\label{2.034}
\end{align}
On the one hand, by using \eqref{2.01}, we have
\begin{eqnarray}
&&\|DU(t)\psi_{k}(D)f\|_{L_{x}^{\infty}L_{t}^{2}}\leq C\|\psi_{k}(D)f\|_{L^{2}}.\label{2.035}
\end{eqnarray}
On the other hand, by using \eqref{2.025}, we have
\begin{eqnarray}
&&\left\|D^{\frac{a}{2}(1-\frac{4}{p})-\frac{1}{p}}U(t)\psi_{k}(D)f\right\|_{L_{x}^{p}L_{t}^{\infty}}\leq C\|\psi_{k}(D)f\|_{L^{2}},\label{2.036}
\end{eqnarray}
By applying the interpolation theorem in mixed Lebesgue spaces \cite[Theorem 1]{BP1961} to \eqref{2.035} and \eqref{2.036}, we have
\begin{eqnarray}
&&\left\|D^{s_{0}}U(t)\psi_{k}(D)f\right\|_{L_{x}^{q}L_{t}^{r}}\leq C\|\psi_{k}(D)f\|_{L^{2}},\label{2.037}
\end{eqnarray}
where
\begin{eqnarray*}
&&\frac{1}{q}=\frac{\theta}{p}+\frac{1-\theta}{\infty},\quad \frac{1}{r}=\frac{\theta}{\infty}+\frac{1-\theta}{2},\\ &&s_{0}=\left(\frac{a}{2}(1-\frac{4}{p})-\frac{1}{p}\right)\theta+1-\theta=\frac{2}{r}-\frac{1}{q}+\frac{a(p-4)}{2q}.
\end{eqnarray*}
By using \eqref{2.034} and \eqref{2.037}, we have
\begin{eqnarray}
&&\left\|D^{s_{0}}U(t)f^{\omega}\right\|_{L_{\omega}^{p_{1}}L_{x}^{q}L_{t}^{r}}\leq C\sqrt{p_{1}}\left\|\|\psi_{k}(D)f\|_{L^{2}}\right\|_{\ell^{2}}\leq C\sqrt{p_{1}}\|f\|_{L^{2}}.\label{2.038}
\end{eqnarray}
By using Lemma 2.5 and \eqref{2.038}, we have that \eqref{2.033} is valid.

The proof of Proposition 2.1 is completed.

\begin{Lemma}\label{lem2.7}
Suppose that  $s\in\R$,\, $\delta>0$,\, $T\in(0,1)$,\, $-\frac{1}{2}<b^{\prime}\leq0\leq b\leq b^{\prime}+1$ and $\eta(x)\in C^{\infty}(\R)$ such that $\supp \eta\subset[0,1]$. Then, we have
\begin{eqnarray}
&&\left\|\eta\left(\frac{t}{T}\right)U(t)f\right\|_{X^{s,\frac{1}{2}+\delta}(\SR\times\SR)}\leq T^{-\delta}\|f\|_{H^{s}(\SR)},\label{2.039}\\
&&\left\|\eta\left(\frac{t}{T}\right)\int_{0}^{t}U(t-\tau)\eta\left(\frac{t}{T}\right)F(\tau)d\tau\right\|_{X^{s,b}}\leq CT^{1+b^{\prime}-b}\|F\|_{X^{s,b^{\prime}}(\SR\times\SR)}.\label{2.040}
\end{eqnarray}
\end{Lemma}

For the proof of Lemma 2.7,  we refer to \cite{B1993, G2002}.

\begin{Lemma}\label{lem2.8}
Suppose that $b>\frac{1}{2}$, $4\leq q_{1}\leq\infty$,\, $2\leq q_{2}\leq\infty$,\, $2\leq r\leq\infty$,\, $\frac{2}{q_{1}}+\frac{1}{r}=\frac{1}{2}$,\, $\alpha=\frac{2}{r}-\frac{1}{q_{1}}$. Then, we have
\begin{eqnarray}
&&\|D^{\alpha}u\|_{L_{x}^{q_{1}}L_{t}^{r}(\SR\times\SR)}\leq C\|u\|_{X^{0,b}(\SR\times\SR)},\label{2.041}\\
&&\|D^{1-\frac{2}{q_{2}}}u\|_{L_{x}^{q_{2}}L_{t}^{2}(\SR\times\SR)}\leq C\|u\|_{X^{0,b(1-\frac{2}{q_{2}})}(\SR\times\SR)}.\label{2.042}
\end{eqnarray}
\end{Lemma}
\noindent{\bf Proof.}
By using the principle of translation invariance of \eqref{2.02}, \eqref{2.03}, we obtain \eqref{2.041}. Next, we apply the interpolation theorem in mixed Lebesgue spaces
\cite[Theorem 1]{BP1961} to the estimates
\begin{eqnarray*}
&&\|u\|_{L_{xt}^{2}}\leq C\|u\|_{X^{0,0}}
\end{eqnarray*}
and
\begin{eqnarray*}
&&\|Du\|_{L_{x}^{\infty}L_{t}^{2}(\SR\times\SR)}\leq C\|u\|_{X^{0,b}(\SR\times\SR)},
\end{eqnarray*}
which yields \eqref{2.042}.

The proof of Lemma 2.8 is completed.

\begin{Lemma}\label{lem2.9}
Suppose that $p, q\geq1$,\, $\varphi(\xi)\in C^{\infty}(\R)$ and $\supp\varphi(\xi)\subset [\frac{1}{4},1]$,\, $\varphi_{j}(\xi)=\varphi(\frac{\xi}{2^{j+1}})$,\,
$\mathscr{F}_{x}\Delta_{j}f=\varphi_{j}(\xi)\mathscr{F}_{x}f(\xi)$.  Then, we have
\begin{eqnarray*}
&&\|\Delta_{j}f\|_{L_{x}^{q}L_{t}^{p}}\leq C\|f\|_{L_{x}^{q}L_{t}^{p}}.
\end{eqnarray*}
\end{Lemma}

For the proof of Lemma 2.9,  we refer to \cite[Lemma 2.15]{YWY2025}.

\begin{Lemma}\label{lem2.10}
Suppose that $f\in L^{p}$,\, $g\in L_{t}^{p_{1}}L_{x}^{q_{1}}$, $1<p,\,q,\,q_{1},\,p_{1}<\infty$.  Then, we have
\begin{eqnarray*}
&&A_{1}\|f\|_{L^{p}}\leq \left\|\|P_{N}f\|_{\ell_{N}^{2}}\right\|_{L^{p}}\leq A_{2}\|f\|_{L^{p}},\\
&&\left\|\|P_{N}f\|_{L_{t}^{p_{1}}L_{x}^{q_{1}}}\right\|_{\ell_{N}^{r}}\leq C\|f\|_{L_{t}^{p_{1}}L_{x}^{q_{1}}},
\end{eqnarray*}
where $A_{1},\, A_{2}>0$,\, $r\geq\max\{p_{1},\, q_{1},\,2\}$.
\end{Lemma}

For the proof of Lemma 2.10,  we refer to \cite[Lemma 2.9]{YYDH2020} and of in \cite[P104, Theorem 5]{S1970}.

\begin{Lemma}\label{lem2.11}(Bernstein inequality in mixed Lebesgue spaces.)
Let $s\in\R$, $d\geq1$, $1\leq r\leq p\leq p_{1}\leq\infty$ and $1\leq q\leq \infty$. Suppose $\varphi(\xi)\in C_{c}^{\infty}(\R^{d})$ is a function supported in the annulus $\{\xi\in\R^{d}:\frac{1}{4}\leq|\xi|\leq1\}$. Then, we have that $\lambda(\xi)=|\xi|^{s}\varphi(\xi)\in C_{c}^{\infty}(\R^{d})$ and $\supp\lambda(\xi)\subset\{\xi\in\R^{d}:\frac{1}{4}\leq|\xi|\leq 1\}$. For $N\in 2^{\mathbf{N}}$, we define the scaled function $\lambda_{N}(\xi)=\left|\frac{\xi}{N}\right|^{s}\varphi\left(\frac{\xi}{N}\right)$ and its inverse Fourier transform
$K_{N}(x)=\check{\lambda}_{N}(x)=N^{d}\check{\lambda}(Nx)$.  Then, we have
\begin{eqnarray}
&&C^{-1}N^{s}N^{d\left(\frac{1}{p_{1}}-\frac{1}{p}\right)}\|P_{N}f\|_{L_{x}^{p_{1}}L_{t}^{q}}\leq\|D^{s}P_{N}f\|_{L_{x}^{p}L_{t}^{q}}\nonumber\\
&&\leq CN^{s}N^{d\left(\frac{1}{r}-\frac{1}{p}\right)}\|P_{N}f\|_{L_{x}^{r}L_{t}^{q}},\label{2.043}
\end{eqnarray}
where
\begin{eqnarray*}
&&P_{N}f\overset{\triangle}{=}\int_{\SR^{d}}e^{ix\cdot\xi}\varphi_{N}(\xi)\hat{f}(\xi,t)d\xi,\\
&&D^{s}P_{N}f=\int_{\SR^{d}}e^{ix\cdot\xi}\varphi_{N}(\xi)|\xi|^{s}\hat{f}(\xi,t)d\xi\\
&&=N^{s}\int_{\SR^{d}}e^{ix\cdot\xi}\lambda_{N}(\xi)\hat{f}(\xi,t)d\xi\\
&&=N^{s}\int_{\SR^{d}}K_{N}(x-y)f(y,t)dy=N^{s}(K_{N}\ast f(\cdot,t))(x).
\end{eqnarray*}
\end{Lemma}

\noindent{\bf Proof.}
For the right-hand side of \eqref{2.043}, we define
\begin{align}
F(x)&=\|D^{s}P_{N}f\|_{L_{t}^{q}}=N^{s}\left(\int_{\SR}|(K_{N}\ast f)|^{q}dt\right)^{\frac{1}{q}}\nonumber\\
&=N^{s}\left(\int_{\SR}\left|\int_{\SR^{d}}K_{N}(x-y) f(y,t)dy\right|^{q}dt\right)^{\frac{1}{q}}.\label{2.044}
\end{align}
Since $q\geq1$, by using  Minkowski inequality, it follows that
\begin{eqnarray}
&&F(x)\leq N^{s}\int_{\SR^{d}}|K_{N}(x-y)| \|f(y,t)\|_{L_{t}^{q}}dy.\label{2.045}
\end{eqnarray}
Define
\begin{eqnarray*}
&&G(x)\overset{\triangle}{=}\|f(x,t)\|_{L_{t}^{q}}.
\end{eqnarray*}
Then, we have
\begin{eqnarray}
&&F(x)\leq N^{s}(K_{N}\ast G)(x).\label{2.046}
\end{eqnarray}
By using Young's inequality and \eqref{2.046}, we have
\begin{align}
\|D^{s}P_{N}f\|_{L_{x}^{p}L_{t}^{q}}&=\|F\|_{L_{x}^{p}}\leq N^{s}\|(K_{N}\ast G)(x)\|_{L_{x}^{p}}\nonumber\\
&\leq CN^{s}\|K_{N}\|_{L_{x}^{\rho}}\|G\|_{L_{x}^{r}},\label{2.047}
\end{align}
where $\frac{1}{p}=\frac{1}{\rho}+\frac{1}{r}-1$.

\noindent Next, we estimate $\|K_{N}\|_{L_{x}^{\rho}}$,
\begin{align}
\|K_{N}\|_{L_{x}^{\rho}}&=\left(\int_{\SR^{d}}|N^{d}\check{\lambda}(Nx)|^{\rho}dx\right)^{\frac{1}{\rho}}\nonumber\\
&=N^{d}\left(\int_{\SR^{d}}|\check{\lambda}(Nx)|^{\rho}dx\right)^{\frac{1}{\rho}}.\label{2.048}
\end{align}
Under the change of variables $z=Nx$, we obtain
\begin{align}
\|K_{N}\|_{L_{x}^{\rho}}&=N^{d}\left(\int_{\SR^{d}}|\check{\lambda}(z)|^{\rho}N^{-d}dz\right)^{\frac{1}{\rho}}\nonumber\\
&=N^{d\left(1-\frac{1}{\rho}\right)}\|\check{\lambda}\|_{L_{z}^{\rho}}.\label{2.049}
\end{align}
Since $\check{\lambda}(x)\in\mathcal{S}$ is a rapidly decaying function, it follows that
\begin{eqnarray}
&&\|K_{N}\|_{L_{x}^{\rho}}\leq CN^{d\left(1-\frac{1}{\rho}\right)}.\label{2.050}
\end{eqnarray}
From \eqref{2.047}, \eqref{2.050}, and the definition of $G(x)$, we have
\begin{align}
\|D^{s}P_{N}f\|_{L_{x}^{p}L_{t}^{q}}&\leq CN^{s}N^{d(1-\frac{1}{\rho})}\|G\|_{L_{x}^{r}}=CN^{s}N^{d\left(\frac{1}{r}-\frac{1}{p}\right)}\|G\|_{L_{x}^{r}}\nonumber\\
&=CN^{s}N^{d\left(\frac{1}{r}-\frac{1}{p}\right)}\|f(x,t)\|_{L_{x}^{r}L_{t}^{q}}.\label{2.051}
\end{align}
Applying \eqref{2.051} to $f=P_{N}f$, and noting that $P_{N}^{2}f=P_{N}f$, we find that
\begin{eqnarray}
&&\|D^{s}P_{N}f\|_{L_{x}^{p}L_{t}^{q}}\leq CN^{s}N^{d\left(\frac{1}{r}-\frac{1}{p}\right)}\|P_{N}f(x,t)\|_{L_{x}^{r}L_{t}^{q}}.\label{2.052}
\end{eqnarray}
For the left-hand side of \eqref{2.043}, it suffices to prove that
\begin{eqnarray}
&&\|P_{N}f\|_{L_{x}^{p_{1}}L_{t}^{q}}\leq CN^{-s}N^{d\left(\frac{1}{p}-\frac{1}{p_{1}}\right)}\|D^{s}P_{N}f\|_{L_{x}^{p}L_{t}^{q}}.\label{2.053}
\end{eqnarray}
By using \eqref{2.052}, we have
\begin{eqnarray}
&&\|D^{-s}P_{N}f\|_{L_{x}^{p_{1}}L_{t}^{q}}\leq CN^{-s}N^{d\left(\frac{1}{p}-\frac{1}{p_{1}}\right)}\|P_{N}f(x,t)\|_{L_{x}^{p}L_{t}^{q}}.\label{2.054}
\end{eqnarray}
Applying \eqref{2.054} to $f=D^{s}f$, and noting that $D^{-s}P_{N}D^{s}f=P_{N}f$, we find that
\begin{eqnarray}
&&\|P_{N}f\|_{L_{x}^{p_{1}}L_{t}^{q}}\leq CN^{-s}N^{d\left(\frac{1}{p}-\frac{1}{p_{1}}\right)}\|D^{s}P_{N}f(x,t)\|_{L_{x}^{p}L_{t}^{q}}.\label{2.055}
\end{eqnarray}
From \eqref{2.055}, we have that \eqref{2.053} is valid. Combining \eqref{2.052} with \eqref{2.053}, we conclude that \eqref{2.043} holds.

The proof of Lemma 2.11 is completed.

\begin{Lemma}\label{lem2.12}
Let $s\in\R$, $1\leq r\leq p\leq\infty$ and $1\leq q\leq \infty$. Suppose $\varphi(\xi)\in C_{c}^{\infty}(\R)$ is a function supported in the annulus $\{\xi\in\R:\frac{1}{4}\leq|\xi|\leq 1\}$.  For $N\in 2^{\mathbf{N}}$, we define the scaled function $\varphi_{N}(\xi)=\varphi\left(\frac{\xi}{N}\right)$.  Then, we have
\begin{eqnarray}
&&\|\langle D\rangle^{s}P_{N}f\|_{L_{x}^{p}L_{t}^{q}}\leq C(1+N)^{s}N^{\frac{1}{r}-\frac{1}{p}}\|P_{N}f\|_{L_{x}^{r}L_{t}^{q}},\label{2.056}
\end{eqnarray}
where
\begin{eqnarray*}
&&P_{N}f\overset{\triangle}{=}\int_{\SR}e^{ix\xi}\varphi_{N}(\xi)\mathscr{F}_{x}f(\xi,t)d\xi,\quad H_{N}(x)\overset{\triangle}{=}\int_{\SR}e^{ix\xi}\varphi(\xi)\langle N\xi\rangle^{s}d\xi,\\
&&\langle D\rangle^{s}P_{N}f=N\int_{\SR}H_{N}(N(x-y))f(y,t)dy=N(H_{N}(N\cdot)\ast f(\cdot,t))(x).
\end{eqnarray*}
\end{Lemma}

\noindent{\bf Proof.}
For \eqref{2.056}, we define
\begin{align}
F(x)&=\left\|\langle D\rangle^{s}P_{N}f\right\|_{L_{t}^{q}}=N\left(\int_{\SR}|(H_{N}(N\cdot)\ast f)|^{q}dt\right)^{\frac{1}{q}}\nonumber\\
&=N\left(\int_{\SR}\left|\int_{\SR}H_{N}(N(x-y)) f(y,t)dy\right|^{q}dt\right)^{\frac{1}{q}}.\label{2.057}
\end{align}
Since $q\geq1$, by using  Minkowski inequality, it follows that
\begin{eqnarray}
&&F(x)\leq N\int_{\SR}|H_{N}(N(x-y))| \|f(y,t)\|_{L_{t}^{q}}dy.\label{2.058}
\end{eqnarray}
Define
\begin{eqnarray*}
&&G(x)\overset{\triangle}{=}\|f(x,t)\|_{L_{t}^{q}}.
\end{eqnarray*}
Then, we have
\begin{eqnarray}
&&F(x)\leq N(H_{N}(N\cdot)\ast G)(x).\label{2.059}
\end{eqnarray}
By using Young's inequality and \eqref{2.059}, we have
\begin{align}
\left\|\langle D\rangle^{s}P_{N}f\right\|_{L_{x}^{p}L_{t}^{q}}&=\|F\|_{L_{x}^{p}}\leq N\|(H_{N}(N\cdot)\ast G)(x)\|_{L_{x}^{p}}\nonumber\\
&\leq CN\|H_{N}(Nx)\|_{L_{x}^{\rho}}\|G\|_{L_{x}^{r}},\label{2.060}
\end{align}
where $\frac{1}{p}=\frac{1}{\rho}+\frac{1}{r}-1$.

\noindent Next, we estimate $\|H_{N}(Nx)\|_{L_{x}^{\rho}}$. Noting that
\begin{eqnarray}
&&\langle x\rangle^{2}H_{N}(x)=\int_{\SR}(I-\partial_{\xi}^{2})\left(e^{ix\xi}\right)\varphi(\xi)\langle N\xi\rangle^{s}d\xi\nonumber\\
&&=\int_{\SR}e^{ix\xi}(I-\partial_{\xi}^{2})\left(\varphi(\xi)\langle N\xi\rangle^{s}\right)d\xi\nonumber\\
&&=\int_{\SR}e^{ix\xi}\left(\varphi(\xi)\langle N\xi\rangle^{s}\right)d\xi-\int_{\SR}e^{ix\xi}\partial_{\xi}^{2}\left(\varphi(\xi)\langle N\xi\rangle^{s}\right)d\xi\nonumber\\
&&=\int_{\SR}e^{ix\xi}\left(\varphi(\xi)-\partial_{\xi}^{2}\varphi(\xi)\right)\langle N\xi\rangle^{s}d\xi-2sN^{2}\int_{\SR}e^{ix\xi}(\partial_{\xi}\varphi(\xi))\xi\langle N\xi\rangle^{s-2}d\xi\nonumber\\
&&\quad-sN^{2}\int_{\SR}e^{ix\xi}\varphi(\xi)\langle N\xi\rangle^{s-2}d\xi-s(s-2)N^{4}\int_{\SR}e^{ix\xi}\varphi(\xi)\xi^{2}\langle N\xi\rangle^{s-4}d\xi.\label{2.061}
\end{eqnarray}
Since
\begin{eqnarray}
&&\left|\langle N\xi\rangle^{s}\right|+2|s|N^{2}\left|\xi\langle N\xi\rangle^{s-2}\right|+|s|N^{2}\left|\langle N\xi\rangle^{s-2}\right|+|s(s-2)|N^{4}\left|\xi^{2}\langle N\xi\rangle^{s-4}\right|\nonumber\\
&&\leq CN^{s}, N>1, \xi\in\supp\varphi(\xi),\label{2.062}\\
&&\left|\langle N\xi\rangle^{s}\right|+2|s|N^{2}\left|\xi\langle N\xi\rangle^{s-2}\right|+|s|N^{2}\left|\langle N\xi\rangle^{s-2}\right|+|s(s-2)|N^{4}\left|\xi^{2}\langle N\xi\rangle^{s-4}\right|\nonumber\\
&&\leq C,\quad N\leq1, \xi\in\supp\varphi(\xi),\label{2.063}
\end{eqnarray}
we have
\begin{eqnarray}
&&\langle x\rangle^{2}|H_{N}(x)|\leq C(1+N)^{s}.\label{2.064}
\end{eqnarray}
It follows from \eqref{2.064} that
\begin{eqnarray}
&&N\|H_{N}(Nx)\|_{L_{x}^{\rho}}\leq C(1+N)^{s}N^{1-\frac{1}{\rho}}.\label{2.065}
\end{eqnarray}
From \eqref{2.060}, \eqref{2.065}, and the definition of $G(x)$, we have
\begin{align}
\|\langle D\rangle^{s}P_{N}f\|_{L_{x}^{p}L_{t}^{q}}&\leq C(1+N)^{s}N^{1-\frac{1}{\rho}}\|G\|_{L_{x}^{r}}=C(1+N)^{s}N^{\frac{1}{r}-\frac{1}{p}}\|G\|_{L_{x}^{r}}\nonumber\\
&=C(1+N)^{s}N^{\frac{1}{r}-\frac{1}{p}}\|f(x,t)\|_{L_{x}^{r}L_{t}^{q}}.\label{2.066}
\end{align}
Applying \eqref{2.066} to $f=P_{N}f$, and noting that $P_{N}^{2}f=P_{N}f$, we find that
\begin{eqnarray}
&&\|\langle D\rangle^{s}P_{N}f\|_{L_{x}^{p}L_{t}^{q}}\leq C(1+N)^{s}N^{\frac{1}{r}-\frac{1}{p}}\|P_{N}f(x,t)\|_{L_{x}^{r}L_{t}^{q}}.\label{2.067}
\end{eqnarray}

The proof of Lemma 2.12 is completed.

\bigskip

\bigskip

\section{Multilinear estimates}

\setcounter{equation}{0}

\setcounter{Theorem}{0}

\setcounter{Lemma}{0}

\setcounter{Proposition}{0}

\setcounter{section}{3}
In order to prove Theorem 1.1, we introduce the following mixed Lebesgue spaces
\begin{eqnarray*}
&&\|f\|_{Y_{N}(\SR)}=\left\|D^{\frac{4-k-8s_{1}}{4k}}\langle D\rangle^{s}P_{N}f\right\|_{L_{x}^{\frac{20k}{5k-4(1-2s_{1})}}L_{t}^{\frac{5k}{2-4s_{1}}}}
+\left\|D^{2s_{1}}\langle D\rangle^{s} P_{N}f\right\|_{L_{x}^{\frac{5}{1-2s_{1}}}L_{t}^{\frac{10}{1+8s_{1}}}}\\
&&\quad+\left\|D^{\frac{4(1+s_{1})-k}{4k}}\langle D\rangle^{s}P_{N}f\right\|_{L_{x}^{\frac{20k}{5k-4(1+s_{1})}}L_{t}^{\frac{5k}{2(1+s_{1})}}}
+\left\|D^{\frac{2s_{1}+a(k-4(1-s_{1}))}{2(k+1)}} \langle D\rangle^{s} P_{N}f\right\|_{L_{x}^{\frac{k+1}{1-s_{1}}}L_{t}^{2(k+1)}}\\
&&\quad+\left\|D^{\frac{a(k-4(1+s_{1}))}{2k}}\langle D\rangle^{s}P_{N}f\right\|_{L_{x}^{\frac{5k}{4(1+s_{1})}}L_{t}^{\frac{5k}{2(1+s_{1})}}}+\left\|D^{\frac{4-k-12s_{1}}{4k}}\langle D\rangle^{s}P_{N}f\right\|_{L_{x}^{\frac{20k}{5k-4+12s_{1}}}L_{t}^{\frac{5k}{2-6s_{1}}}}\\
&&\quad+\left\|D^{(4a+1)s_{1}}\langle D\rangle^{s}P_{N}f\right\|_{L_{x}^{\frac{5}{1-9s_{1}}}L_{t}^{\frac{10}{1-4s_{1}}}}
+\left\| D^{-\frac{2s_{1}}{k}+\frac{(k-4+4s_{1})a}{2k}}\langle D\rangle^{s}P_{N}f\right\|
_{L_{x}^{\frac{5k}{4-2s_{1}}}L_{t}^{\frac{5k}{2-6s_{1}}}},\\
&&\|f\|_{Y(\SR)}=\|\|f\|_{Y_{N}(\SR)}\|_{\ell_{N}^{2}}=\left(\sum_{N\in2^{\bf{N}}}\|f\|_{Y_{N}(\SR)}^{2}\right)^{\frac{1}{2}}.
\end{eqnarray*}
where $k\geq 5$,\, $s\in\R$,\, $0<s_{1}<\frac{1}{9}$, and $a\in\mathbf{N}$, such that
\begin{eqnarray*}
&a>\max\left\{\frac{2(k+1)(\sigma+2s_{1}-s)-2s_{1}}{k-4+4s_{1}},\, -\frac{2ks}{k-4-4s_{1}},\,\frac{\sigma-s}{4s_{1}}+\frac{1}{4},\,
-\frac{2(ks-2s_{1})}{k-4+4s_{1}} \right\}.
\end{eqnarray*}
Firstly, we prove the multilinear estimates, which play an important role in proving Theorem 1.1.

\begin{Proposition}\label{pro3.1}
Suppose that $k\geq 5$,\, $s\in\R$,\, $0<s_{1}<\frac{1}{9}$,\, \,$0<\delta<\frac{s_{1}}{k+4-2s_{1}}$,\, $b=\frac{1}{2}+\delta$,\, $b_{1}=-\frac{1}{2}+(k+3)\delta$, $\sigma>\frac{1}{2}-\frac{2}{k}+\frac{4s_{1}}{k}$ and $a\in\mathbf{N}$, such that
\begin{eqnarray*}
&a>\max\left\{\frac{2(k+1)(\sigma+2s_{1}-s)-2s_{1}}{k-4+4s_{1}},\, -\frac{2ks}{k-4-4s_{1}},\,\frac{\sigma-s}{4s_{1}}+\frac{1}{4},\,
-\frac{2(ks-2s_{1})}{k-4+4s_{1}} \right\}.
\end{eqnarray*}
Then,  we have
\begin{eqnarray}
&&\|\partial_{x}(v+z)^{k+1}\|_{X^{\sigma, b_{1}}(\SR\times\SR)}\leq C\left(\|v\|_{X^{\sigma,b}(\SR\times\SR)}^{k+1}+\|z\|_{Y(\SR)}^{k+1}\right).\label{3.01}
\end{eqnarray}

\end{Proposition}
\noindent{\bf Proof.}
According to the duality idea, to prove \eqref{3.01}, we only need to prove
\begin{eqnarray}
&&\left|\int_{\SR\times\SR}\langle D\rangle^{\sigma}\partial_{x}[(v+z)^{k+1}]hdxdt\right|\leq C\|h\|_{X^{0,-b_{1}}}
\left(\|v\|_{X^{\sigma,b}}^{k+1}+\|z\|_{Y(\tiny\bf{R})}^{k+1}\right).\label{3.02}
\end{eqnarray}
By using H\"{o}lder inequality and \eqref{2.042}, $s_{1}\leq\frac{1}{2}$,\, $b(1-2s_{1})<-b_{1}$, we have
\begin{align}
&\quad\left|\int_{\SR\times\SR}\langle D\rangle^{\sigma}\partial_{x}[(v+z)^{k+1}]hdxdt\right|\nonumber\\
&\leq C
\left\|D^{2s_{1}}\langle D\rangle^{\sigma}(v+z)^{k+1}\right\|_{L_{x}^{\frac{1}{1-s_{1}}}L_{t}^{2}}\left\|D^{1-2s_{1}}h\right\|_{L_{x}^{\frac{1}{s_{1}}}L_{t}^{2}}\nonumber\\
&\leq C\left\|D^{2s_{1}}\langle D\rangle^{\sigma}(v+z)^{k+1}\right\|_{L_{x}^{\frac{1}{1-s_{1}}}L_{t}^{2}}\|h\|_{X^{0,-b_{1}}}.\label{3.03}
\end{align}
From \eqref{3.03}, to prove \eqref{3.02}, we only need to prove
\begin{eqnarray}
&&\left\|D^{2s_{1}}\langle D\rangle^{\sigma}(v+z)^{k+1}\right\|_{L_{x}^{\frac{1}{1-s_{1}}}L_{t}^{2}}\leq C\left(\|v\|_{X^{\sigma,b}}^{k+1}+\|z\|_{Y(\tiny\bf{R})}^{k+1}\right).\label{3.04}
\end{eqnarray}
In order to prove \eqref{3.04}, we introduce the dyadic decompositions of $v$ and $z$. We denote $P_{N_{i}}v=v_{i}$, $P_{N_{i}}z=z_{i}$, then, we see that $\supp \mathscr{F}_{x}v_{i}\subset\{\xi||\xi|\sim N_{i}\}$, and  $\supp\mathscr{F}_{x}z_{i}\subset\{\xi||\xi|\sim N_{i}\}$. Next, we will consider the following three cases:

\noindent{\bf Case 1:} The terms on the left-hand side of \eqref{3.04} contain only $v's$; that is, $v_{1}\cdot v_{2}\cdot\ldots \cdot v_{k+1}$. We denote $\sum\limits_{N_{1}, N_{2},\ldots, N_{k+1}}=\sum$.
Without loss of generality, we assume that $N_{1}\leq N_{2}\leq \ldots\leq N_{k+1}$, then,
we consider the following two cases: {\bf Case 1A:} $N_{k+1}\leq 1$ and {\bf Case 1B:} $N_{k+1}>1$, respectively.

\noindent{\bf Case 1A:} For $N_{k+1}\leq 1$, by using \eqref{2.041} and H\"{o}lder inequality, Lemmas 2.11, 2.12, we have
\begin{align}
&\quad\left\|D^{2s_{1}}\langle D\rangle^{\sigma}\left(\prod_{i=1}^{k+1}v_{i}\right)\right\|_{L_{x}^{\frac{1}{1-s_{1}}}L_{t}^{2}}\leq CN_{k+1}^{2s_{1}}\left\|\left(\prod_{i=1}^{k+1}v_{i}\right)\right\|_{L_{x}^{\frac{1}{1-s_{1}}}L_{t}^{2}}\nonumber\\
&\leq C\left(\prod_{i=1}^{k}\|v_{i}\|_{L_{x}^{\frac{5k}{4-3s_{1}}}L_{t}^{\frac{5k}{2-4s_{1}}}}\right)N_{k+1}^{2s_{1}}\left\|v_{k+1}\right\|
_{L_{x}^{\frac{5}{1-2s_{1}}}L_{t}^{\frac{10}{1+8s_{1}}}}\nonumber\\
&\leq
C\left(\prod_{i=1}^{k}\|v_{i}\|_{L_{x}^{\frac{5k}{4-3s_{1}}}L_{t}^{\frac{5k}{2-4s_{1}}}}\right)
\left\|D^{2s_{1}}v_{k+1}\right\|_{L_{x}^{\frac{5}{1-2s_{1}}}L_{t}^{\frac{10}{1+8s_{1}}}}\nonumber\\
&\leq C\left(\prod_{i=1}^{k}N_{i}^{\frac{k-4(1-s_{1})}{4k}}
\left\|D^{\frac{4-k-8s_{1}}{4k}}D^{-\frac{4-k-8s_{1}}{4k}}v_{i}\right\|_{L_{x}^{\frac{20k}{5k-4(1-2s_{1})}}L_{t}^{\frac{5k}{2-4s_{1}}}}\right)
\|v_{k+1}\|_{X^{0,b}}\nonumber\\
&\leq C\left(\prod_{i=1}^{k}N_{i}^{\frac{k-4(1-s_{1})}{4k}-\frac{4-k-8s_{1}}{4k}}
\left\|D^{\frac{4-k-8s_{1}}{4k}}v_{i}\right\|_{L_{x}^{\frac{20k}{5k-4(1-2s_{1})}}L_{t}^{\frac{5k}{2-4s_{1}}}}\right)
\|v_{k+1}\|_{X^{\sigma,b}}\nonumber\\
&\leq C\left(\prod_{i=1}^{k}N_{i}^{\frac{k-4(1-s_{1})}{4k}-\frac{4-k-8s_{1}}{4k}}\|v_{i}\|_{X^{0,b}}\right)
\|v_{k+1}\|_{X^{\sigma,b}}.\label{3.05}
\end{align}
where we use the following facts that $N_{i}\leq 1$,
\begin{eqnarray*}
&&\frac{2-4s_{1}}{5}+\frac{1+8s_{1}}{10}=\frac{1}{2},\quad \frac{1+8s_{1}}{5}-\frac{1-2s_{1}}{5}=2s_{1},\\
&& \frac{4(1-2s_{1})}{5k}-\frac{5k-4(1-2s_{1})}{20k}=\frac{20-5k-40s_{1}}{20k}=\frac{4-k-8s_{1}}{4k},
\end{eqnarray*}
and
\begin{eqnarray*}
&& \frac{k-4(1-s_{1})}{4k}-\frac{4-k-8s_{1}}{4k}=\frac{1}{2}-\frac{2}{k}+\frac{3s_{1}}{k}>0.
\end{eqnarray*}
By using \eqref{3.06} and Lemma 2.10, we have
\begin{eqnarray}
&&\sum\left\|D^{2s_{1}}\langle D\rangle^{\sigma}\left(\prod_{i=1}^{k+1}v_{i}\right)\right\|_{L_{x}^{\frac{1}{1-s_{1}}}L_{t}^{2}}
\leq C\|v\|_{X^{\sigma,b}}^{k+1}.\label{3.06}
\end{eqnarray}

\noindent{\bf Case 1B:} For $N_{k+1}>1$, by using \eqref{2.041} and H\"{o}lder inequality, Lemmas 2.11, 2.12, we have
\begin{align}
&\quad\left\|D^{2s_{1}}\langle D\rangle^{\sigma}\left(\prod_{i=1}^{k+1}v_{i}\right)\right\|_{L_{x}^{\frac{1}{1-s_{1}}}L_{t}^{2}}\leq CN_{k+1}^{2s_{1}}N_{k+1}^{\sigma}\left\|\left(\prod_{i=1}^{k+1}v_{i}\right)\right\|_{L_{x}^{\frac{1}{1-s_{1}}}L_{t}^{2}}\nonumber\\
&\leq
C\left(\prod_{i=1}^{k}\|v_{i}\|_{L_{x}^{\frac{5k}{4-2s_{1}}}L_{t}^{\frac{5k}{2-6s_{1}}}}N_{k+1}^{2s_{1}}\right)N_{k+1}^{\sigma}\left\|v_{k+1}\right\|
_{L_{x}^{\frac{5}{1-3s_{1}}}L_{t}^{\frac{10}{1+12s_{1}}}}\nonumber\\
&\leq C\left(\prod_{i=1}^{k}\|v_{i}\|_{L_{x}^{\frac{5k}{4-2s_{1}}}L_{t}^{\frac{5k}{2-6s_{1}}}}\right)N_{k+1}^{2s_{1}}N_{k+1}^{\sigma}\left\|D^{-3s_{1}}\langle D\rangle^{-\sigma}D^{3s_{1}}\langle D\rangle^{\sigma}v_{k+1}\right\|
_{L_{x}^{\frac{5}{1-3s_{1}}}L_{t}^{\frac{10}{1+12s_{1}}}}\nonumber\\
&\leq C\left(\prod_{i=1}^{k}\|v_{i}\|_{L_{x}^{\frac{5k}{4-2s_{1}}}L_{t}^{\frac{5k}{2-6s_{1}}}}\right)N_{k+1}^{2s_{1}}N_{k+1}^{\sigma}N_{k+1}^{-3s_{1}}N_{k+1}^{-\sigma}
\left\|D^{3s_{1}}\langle D\rangle^{\sigma}v_{k+1}\right\|
_{L_{x}^{\frac{5}{1-3s_{1}}}L_{t}^{\frac{10}{1+12s_{1}}}}\nonumber\\
&\leq
C\left(\prod_{i=1}^{k}\|v_{i_{2}}\|_{L_{x}^{\frac{5k}{4-2s_{1}}}L_{t}^{\frac{5k}{2-6s_{1}}}}\right)N_{k+1}^{-s_{1}}\left\|D^{3s_{1}}\langle D\rangle^{\sigma}v_{k+1}\right\|
_{L_{x}^{\frac{5}{1-3s_{1}}}L_{t}^{\frac{10}{1+12s_{1}}}}\nonumber\\
&\leq C
\left(\prod_{i=1}^{k} N_{i}^{\frac{5k-20+20s_{1}}{20k}}\left\|D^{\frac{20-5k-60s_{1}}{20k}}D^{-\frac{20-5k-60s_{1}}{20k}}v_{i}\right\|
_{L_{x}^{\frac{20k}{5k-4+12s_{1}}}L_{t}^{\frac{5k}{2-6s_{1}}}}\right)N_{k+1}^{-s_{1}}\left\|v_{k+1}\right\|_{X^{\sigma,b}}\nonumber\\
&\leq C
\left(\prod_{i=1}^{k} N_{i}^{\frac{5k-20+20s_{1}}{20k}-\frac{20-5k-60s_{1}}{20k}}\left\|D^{\frac{20-5k-60s_{1}}{20k}}v_{i}\right\|
_{L_{x}^{\frac{20k}{5k-4+12s_{1}}}L_{t}^{\frac{5k}{2-6s_{1}}}}\right)N_{k+1}^{-s_{1}}\left\|v_{k+1}\right\|_{X^{\sigma,b}}\nonumber\\
&\leq C \left(\prod_{i=1}^{k}N_{i}^{\frac{5k-20+20s_{1}}{20k}-\frac{20-5k-60s_{1}}{20k}}\left\|v_{i}\right\|
_{X^{0,b}}\right)
N_{k+1}^{-s_{1}}\|v_{k+1}\|_{X^{\sigma,b}},\label{3.07}
\end{align}
where we use the following facts that $N_{k+1}>1$,
\begin{eqnarray*}
&&\frac{2(1-3s_{1})}{5}+\frac{1+12s_{1}}{10}=\frac{1}{2},\quad \frac{1+12s_{1}}{5}-\frac{1-3s_{1}}{5}=3s_{1},\\
&&\frac{2(1-3s_{1})}{5k}+\frac{2(5k-4+12s_{1})}{20k}=\frac{1}{2},\quad\frac{10}{1+12s_{1}}\geq2\Leftrightarrow s_{1}\leq\frac{1}{3},\\
&& \frac{4(1-3s_{1})}{5k}-\frac{5k-4+12s_{1}}{20k}=\frac{20-5k-60s_{1}}{20k},
\end{eqnarray*}
and
\begin{eqnarray*}
&& \frac{5k-20+20s_{1}}{20k}-\frac{20-5k-60s_{1}}{20k}=\frac{1}{2}-\frac{2}{k}+\frac{4s_{1}}{k}\in(0,\sigma).
\end{eqnarray*}
By using \eqref{3.07} and Lemma 2.10, we have
\begin{eqnarray}
&&\sum\left\|D^{2s_{1}}\langle D\rangle^{\sigma}\left(\prod_{i=1}^{k+1}v_{i}\right)\right\|_{L_{x}^{\frac{1}{1-s_{1}}}L_{t}^{2}}
\leq C\|v\|_{X^{\sigma,b}}^{k+1}.\label{3.08}
\end{eqnarray}

\noindent{\bf Case 2:}
The terms on the left-hand side of \eqref{3.04} contain only $z's$; that is, $z_{1}\cdot z_{2}\cdot\ldots\cdot z_{k+1}$. We denote $\sum\limits_{N_{1}, N_{2},\ldots, N_{k+1}}=\sum$.
Without loss of generality, we assume that $N_{1}\leq N_{2}\leq \ldots\leq N_{k+1}$, then,
we consider the following two cases: {\bf Case 2A:} $N_{k+1}\leq 1$ and {\bf Case 2B:} $N_{k+1}>1$, respectively.

\noindent{\bf Case 2A:} For $N_{k+1}\leq 1$, by using H\"{o}lder inequality, Lemmas 2.11, 2.12 and a proof similar to \eqref{3.05}, we have
\begin{align}
&\quad\sum\left\|D^{2s_{1}}\langle D\rangle^{\sigma}\left(\prod_{i=1}^{k+1}z_{i}\right)\right\|_{L_{x}^{\frac{1}{1-s_{1}}}L_{t}^{2}}\leq C\sum N_{k+1}^{2s_{1}}\left\|\left(\prod_{i=1}^{k+1}z_{i}\right)\right\|_{L_{x}^{\frac{1}{1-s_{1}}}L_{t}^{2}}\nonumber\\
&\leq
C\sum\left[\left(\prod_{i=1}^{k}\|z_{i}\|_{L_{x}^{\frac{5k}{4-3s_{1}}}L_{t}^{\frac{5k}{2-4s_{1}}}}\right)N_{k+1}^{2s_{1}}\left\|z_{k+1}\right\|
_{L_{x}^{\frac{5}{1-2s_{1}}}L_{t}^{\frac{10}{1+8s_{1}}}}\right]\nonumber\\
&\leq
C\sum\left[\left(\prod_{i=1}^{k}\|z_{i}\|_{L_{x}^{\frac{5k}{4-3s_{1}}}L_{t}^{\frac{5k}{2-4s_{1}}}}\right)
\left\|D^{2s_{1}}\langle D\rangle^{s} z_{k+1}\right\|_{L_{x}^{\frac{5}{1-2s_{1}}}L_{t}^{\frac{10}{1+8s_{1}}}}\right]\nonumber\\
&\leq C\sum\left[\left(\prod_{i=1}^{k}N_{i}^{\frac{k-4(1-s_{1})}{4k}}
\left\|D^{\frac{4-k-8s_{1}}{4k}}D^{-\frac{4-k-8s_{1}}{4k}}z_{i}\right\|_{L_{x}^{\frac{20k}{5k-4(1-2s_{1})}}L_{t}^{\frac{5k}{2-4s_{1}}}}\right)\right.\nonumber\\
&\quad\left.\times\left\|D^{2s_{1}}\langle D\rangle^{s} z_{k+1}\right\|_{L_{x}^{\frac{5}{1-2s_{1}}}L_{t}^{\frac{10}{1+8s_{1}}}}\right]\nonumber\\
&\leq C\sum\left[\left(\prod_{i=1}^{k}N_{i}^{\frac{k-4(1-s_{1})}{4k}-\frac{4-k-8s_{1}}{4k}}
\left\|D^{\frac{4-k-8s_{1}}{4k}}\langle D\rangle^{s} z_{i}\right\|_{L_{x}^{\frac{20k}{5k-4(1-2s_{1})}}L_{t}^{\frac{5k}{2-4s_{1}}}}\right)\right.\nonumber\\
&\quad\left.\times\left\|D^{2s_{1}}\langle D\rangle^{s} z_{k+1}\right\|_{L_{x}^{\frac{5}{1-2s_{1}}}L_{t}^{\frac{10}{1+8s_{1}}}}\right]\nonumber\\
&\leq C\|z\|_{Y(\SR)}^{k+1},\label{3.09}
\end{align}
where we use the following facts that $N_{i}\leq 1$,
\begin{eqnarray*}
&&\frac{2-4s_{1}}{5}+\frac{1+8s_{1}}{10}=\frac{1}{2},\quad \frac{1+8s_{1}}{5}-\frac{1-2s_{1}}{5}=2s_{1},\\
&& \frac{4(1-2s_{1})}{5k}-\frac{5k-4(1-2s_{1})}{20k}=\frac{20-5k-40s_{1}}{20k}=\frac{4-k-8s_{1}}{4k},
\end{eqnarray*}
and
\begin{eqnarray*}
&& \frac{k-4(1-s_{1})}{4k}-\frac{4-k-8s_{1}}{4k}=\frac{1}{2}-\frac{2}{k}+\frac{3s_{1}}{k}>0.
\end{eqnarray*}

\noindent{\bf Case 2B:} For $N_{k+1}> 1$, we consider $N_{1}>1$ and $N_{1}\leq1$, respectively.

When $N_{1}>1$, by using  Lemmas 2.11, 2.12 and  H\"{o}lder inequality, we have
\begin{align}
&\quad\left\|D^{2s_{1}}\langle D\rangle^{\sigma}\left(\prod_{i=1}^{k+1}z_{i}\right)\right\|_{L_{x}^{\frac{1}{1-s_{1}}}L_{t}^{2}}
\leq CN_{k+1}^{2s_{1}}N_{k+1}^{\sigma}\left\|\left(\prod_{i=1}^{k+1}z_{i}\right)\right\|_{L_{x}^{\frac{1}{1-s_{1}}}L_{t}^{2}}\nonumber\\
&\leq
C\left(\prod_{i=1}^{k}\|z_{i}\|_{L_{x}^{\frac{k+1}{1-s_{1}}}L_{t}^{2(k+1)}}\right)
N_{k+1}^{2s_{1}}N_{k+1}^{\sigma}\left\|z_{k+1}\right\|_{L_{x}^{\frac{k+1}{1-s_{1}}}L_{t}^{2(k+1)}}\nonumber\\
&\leq
C\left(\prod_{i=1}^{k}\|z_{i}\|_{L_{x}^{\frac{k+1}{1-s_{1}}}L_{t}^{2(k+1)}}\right)
\left\|D^{2s_{1}}\langle D\rangle^{\sigma}z_{k+1}\right\|_{L_{x}^{\frac{k+1}{1-s_{1}}}L_{t}^{2(k+1)}}\nonumber\\
&\leq C\left(\prod_{i=1}^{k}
\left\|D^{\frac{2s_{1}+a(k-4(1-s_{1}))}{2(k+1)}}\langle D\rangle^{s}\langle D\rangle^{-s} D^{-\frac{2s_{1}+a(k-4(1-s_{1}))}{2(k+1)}} z_{i}\right\|_{L_{x}^{\frac{k+1}{1-s_{1}}}L_{t}^{2(k+1)}}\right)\nonumber\\
&\quad\times\left\|D^{\frac{2s_{1}+a(k-4(1-s_{1}))}{2(k+1)}}\langle D\rangle^{s}\langle D\rangle^{-s}D^{-\frac{2s_{1}+a(k-4(1-s_{1}))}{2(k+1)}}D^{2s_{1}}\langle D\rangle^{\sigma}z_{k+1}\right\|_{L_{x}^{\frac{k+1}{1-s_{1}}}L_{t}^{2(k+1)}}\nonumber\\
&\leq C\left(\prod_{i=1}^{k}N_{i}^{-s-\frac{2s_{1}+a(k-4(1-s_{1}))}{2(k+1)}}\left\|D^{\frac{2s_{1}+a(k-4(1-s_{1}))}{2(k+1)}}\langle D\rangle^{s} z_{i}\right\|_{L_{x}^{\frac{k+1}{1-s_{1}}}L_{t}^{2(k+1)}}\right)\nonumber\\
&\quad\times N_{k+1}^{2s_{1}+\sigma-s-\frac{2s_{1}+a(k-4(1-s_{1}))}{2(k+1)}}
\left\|D^{\frac{2s_{1}+a(k-4(1-s_{1}))}{2(k+1)}} \langle D\rangle^{s} z_{k+1}\right\|_{L_{x}^{\frac{k+1}{1-s_{1}}}L_{t}^{2(k+1)}}.\label{3.010}
\end{align}
By using \eqref{3.010}, H\"{o}lder inequality, and $a\in\mathbf{N}$ with
\begin{eqnarray*}
&&a>\frac{2(k+1)(\sigma+2s_{1}-s)-2s_{1}}{k-4+4s_{1}},
\end{eqnarray*}
it follows that the powers of $N_{i}(1\leq i\leq k+1)$ in \eqref{3.010} are negative. Since $N_{i}>1,\,1\leq i\leq k+1$, we have
\begin{eqnarray}
&&\sum\left\|D^{2s_{1}}\langle D\rangle^{\sigma}\left(\prod_{i=1}^{k+1}z_{i}\right)\right\|_{L_{x}^{\frac{1}{1-s_{1}}}L_{t}^{2}}
\leq C\|z\|_{Y(\SR)}^{k+1}.\label{3.011}
\end{eqnarray}

When $N_{1}\leq1$, then there exists a $\tilde{k}$ such that $N_{\tilde{k}}\leq1$ and $N_{\tilde{k}+1}>1$. By using H\"{o}lder inequality and Lemmas 2.11, 2.12, we have
\begin{align}
&\quad\left\|D^{2s_{1}}\langle D\rangle^{\sigma}\left(\prod_{i=1}^{k+1}z_{i}\right)\right\|_{L_{x}^{\frac{1}{1-s_{1}}}L_{t}^{2}}
\leq CN_{k+1}^{2s_{1}}N_{k+1}^{\sigma}\left\|\left(\prod_{i=1}^{k+1}z_{i}\right)\right\|_{L_{x}^{\frac{1}{1-s_{1}}}L_{t}^{2}}\nonumber\\
&\leq
C\left(\prod_{i=1}^{k}\|z_{i}\|_{L_{x}^{\frac{5k}{4(1+s_{1})}}L_{t}^{\frac{5k}{2(1+s_{1})}}}\right)
N_{k+1}^{2s_{1}}N_{k+1}^{\sigma}\left\|z_{k+1}\right\|_{L_{x}^{\frac{5}{1-9s_{1}}}L_{t}^{\frac{10}{1-4s_{1}}}}\nonumber\\
&\leq
C\left(\prod_{i=1}^{\tilde{k}}\|z_{i}\|_{L_{x}^{\frac{5k}{4(1+s_{1})}}L_{t}^{\frac{5k}{2(1+s_{1})}}}\right)
\nonumber\\
&\quad\times\left(\prod_{i=\tilde{k}+1}^{k}\|z_{i}\|_{L_{x}^{\frac{5k}{4(1+s_{1})}}L_{t}^{\frac{5k}{2(1+s_{1})}}}\right)
\left\|D^{2s_{1}}\langle D\rangle^{\sigma}z_{k+1}\right\|_{L_{x}^{\frac{5}{1-9s_{1}}}L_{t}^{\frac{10}{1-4s_{1}}}}\nonumber\\
&\leq
C\left(\prod_{i=1}^{\tilde{k}}N_{i}^{\frac{k-4-4s_{1}}{4k}}\|z_{i}\|_{L_{x}^{\frac{20k}{5k-4(1+s_{1})}}L_{t}^{\frac{5k}{2(1+s_{1})}}}\right)\nonumber\\
&\quad\times\left(\prod_{i=\tilde{k}+1}^{k}\left\|D^{\frac{a(k-4(1+s_{1}))}{2k}}\langle D\rangle^{s}\langle D\rangle^{-s} D^{-\frac{a(k-4(1+s_{1}))}{2k}}z_{i}\right\|_{L_{x}^{\frac{5k}{4(1+s_{1})}}L_{t}^{\frac{5k}{2(1+s_{1})}}}\right)\nonumber\\
&\quad\times \left\|D^{(4a+1)s_{1}}\langle D\rangle^{s}\langle D\rangle^{-s}D^{-(4a+1)s_{1}}D^{2s_{1}}\langle D\rangle^{\sigma}z_{k+1}\right\|_{L_{x}^{\frac{5}{1-9s_{1}}}L_{t}^{\frac{10}{1-4s_{1}}}}\nonumber\\
&\leq
C\left(\prod_{i=1}^{\tilde{k}}N_{i}^{\frac{k-4-4s_{1}}{4k}}
\left\|D^{\frac{4(1+s_{1})-k}{4k}}D^{-\frac{4(1+s_{1})-k}{4k}}z_{i}\right\|_{L_{x}^{\frac{20k}{5k-4(1+s_{1})}}L_{t}^{\frac{5k}{2(1+s_{1})}}}\right)\nonumber\\
&\quad\times\left(\prod_{i=\tilde{k}+1}^{k}N_{i}^{-s-\frac{a(k-4(1+s_{1}))}{2k}}\left\|D^{\frac{a(k-4(1+s_{1}))}{2k}}\langle D\rangle^{s} z_{i}\right\|_{L_{x}^{\frac{5k}{4(1+s_{1})}}L_{t}^{\frac{5k}{2(1+s_{1})}}}\right)\nonumber\\
&\quad\times N_{k+1}^{\sigma-s-(4a-1)s_{1}}\left\|D^{(4a+1)s_{1}}\langle D\rangle^{s}z_{k+1}\right\|_{L_{x}^{\frac{5}{1-9s_{1}}}L_{t}^{\frac{10}{1-4s_{1}}}}\nonumber\\
&\leq
C\left(\prod_{i=1}^{\tilde{k}}N_{i}^{\frac{k-4(1+s_{1})}{4k}-\frac{4(1+s_{1})-k}{4k}}\left\|D^{\frac{4(1+s_{1})-k}{4k}}\langle D\rangle^{s} z_{i}\right\|_{L_{x}^{\frac{20k}{5k-4(1+s_{1})}}L_{t}^{\frac{5k}{2(1+s_{1})}}}\right)\nonumber\\
&\quad\times\left(\prod_{i=\tilde{k}+1}^{k}N_{i}^{-s-\frac{a(k-4(1+s_{1}))}{2k}}\left\|D^{\frac{a(k-4(1+s_{1}))}{2k}}\langle D\rangle^{s} z_{i}\right\|_{L_{x}^{\frac{5k}{4(1+s_{1})}}L_{t}^{\frac{5k}{2(1+s_{1})}}}\right)\nonumber\\
&\quad\times N_{k+1}^{\sigma-s-(4a-1)s_{1}}\left\|D^{(4a+1)s_{1}}\langle D\rangle^{s}z_{k+1}\right\|_{L_{x}^{\frac{5}{1-9s_{1}}}L_{t}^{\frac{10}{1-4s_{1}}}},\label{3.012}
\end{align}
where we use the following facts
\begin{eqnarray*}
&& \frac{4(1+s_{1})}{5k}-\frac{5k-4(1+s_{1})}{20k}=\frac{20(1+s_{1})-5k}{20k}=\frac{4(1+s_{1})-k}{4k},\\
&& \frac{k-4(1+s_{1})}{4k}-\frac{4(1+s_{1})-k}{4k}=\frac{1}{2}-\frac{2}{k}-\frac{2s_{1}}{k}>0.
\end{eqnarray*}
By using \eqref{3.012}, H\"{o}lder inequality and $a\in\mathbf{N}$ with
\begin{eqnarray*}
&&a>\max\left\{ -\frac{2ks}{k-4-4s_{1}},\,\frac{\sigma-s}{4s_{1}}+\frac{1}{4}\right\},
\end{eqnarray*}
it follows that the powers of $N_{i}(i\geq \tilde{k}+1)$ and $N_{k+1}$ in \eqref{3.012} are negative. Since $N_{i}(i\geq \tilde{k}+1)$ and $N_{k+1}$ are greater than 1, we have
\begin{eqnarray}
&&\sum\left\|D^{2s_{1}}\langle D\rangle^{\sigma}\left(\prod_{i=1}^{k+1}z_{i}\right)\right\|_{L_{x}^{\frac{1}{1-s_{1}}}L_{t}^{2}}
\leq  C\|z\|_{Y(\tiny\bf{R})}^{k+1}.\label{3.013}
\end{eqnarray}

\noindent{\bf Case 3:}
The left-hand side of \eqref{3.04} contains both $v$ and $z$; that is, $z_{1}\cdot z_{2}\cdot\ldots z_{j}\cdot v_{j+1}\ldots\cdot v_{k+1}$. We denote
$\sum\limits_{N_{1}, N_{2},\ldots, N_{j}}\left(\sum\limits_{N_{j+1}, N_{j+2},\ldots, N_{k+1}}\right)=\sum$.
Without loss of generality, we assume that $N_{1}\leq N_{2}\leq \ldots\leq N_{j}$  and $N_{j+1}\leq N_{j+2}\leq \ldots\leq N_{k+1}$,
then, we consider the following two cases: {\bf Case 3A:} $N_{j}\leq N_{k+1}$ and {\bf Case 3B:} $N_{j}>N_{k+1}$, respectively.

\noindent{\bf Case 3A:} $N_{j}\leq N_{k+1}$.

When $N_{j}\leq 1\leq N_{k+1}$,
by using \eqref{2.041} and H\"{o}lder inequality, Lemmas 2.11, 2.12, we have
\begin{align}
&\quad\left\|D^{2s_{1}}\langle D\rangle^{\sigma}\left(\prod_{m=1}^{j}z_{m}\prod_{n=j+1}^{k+1}v_{n}\right)\right\|_{L_{x}^{\frac{1}{1-s_{1}}}L_{t}^{2}}
\leq CN_{k+1}^{2s_{1}}N_{k+1}^{\sigma}\left\|\left(\prod_{m=1}^{j}z_{m}\prod_{n=j+1}^{k+1}v_{n}\right)
\right\|_{L_{x}^{\frac{1}{1-s_{1}}}L_{t}^{2}}\nonumber\\
&\leq
C\left(\prod_{m=1}^{j}\|z_{m}\|_{L_{x}^{\frac{5k}{4-2s_{1}}}L_{t}^{\frac{5k}{2-6s_{1}}}}\right)
\left(\prod_{n=j+1}^{k}\|v_{n}\|_{L_{x}^{\frac{5k}{4-2s_{1}}}L_{t}^{\frac{5k}{2-6s_{1}}}}\right)N_{k+1}^{2s_{1}}N_{k+1}^{\sigma}\left\|v_{k+1}\right\|
_{L_{x}^{\frac{5}{1-3s_{1}}}L_{t}^{\frac{10}{1+12s_{1}}}}\nonumber\\
&\leq
C\left(\prod_{m=1}^{j}\|z_{m}\|_{L_{x}^{\frac{5k}{4-2s_{1}}}L_{t}^{\frac{5k}{2-6s_{1}}}}\right)
\nonumber\\
&\quad\times\left(\prod_{n=j+1}^{k}\|v_{n}\|_{L_{x}^{\frac{5k}{4-2s_{1}}}L_{t}^{\frac{5k}{2-6s_{1}}}}\right)\left\|D^{2s_{1}}\langle D\rangle^{\sigma}v_{k+1}\right\|
_{L_{x}^{\frac{5}{1-3s_{1}}}L_{t}^{\frac{10}{1+12s_{1}}}}\nonumber\\
&\leq C\left(\prod_{m=1}^{j}N_{m}^{\frac{5k-20+20s_{1}}{20k}}
\left\|D^{\frac{20-5k-60s_{1}}{20k}}D^{-\frac{20-5k-60s_{1}}{20k}}z_{m}\right\|_{L_{x}^{\frac{20k}{5k-4+12s_{1}}}L_{t}^{\frac{5k}{2-6s_{1}}}}\right)\nonumber\\
&\quad\times\left(\prod_{n=j+1}^{k}N_{n}^{\frac{5k-20+20s_{1}}{20k}}\left\|v_{n}\right\|_{L_{x}^{\frac{20k}{5k-4+12s_{1}}}L_{t}^{\frac{5k}{2-6s_{1}}}}\right)
N_{k+1}^{-s_{1}}\left\|v_{k+1}\right\|_{X^{\sigma,b}}\nonumber\\
&\leq C\left(\prod_{m=1}^{j}N_{m}^{\frac{5k-20+20s_{1}}{20k}-\frac{20-5k-60s_{1}}{20k}}
\left\|D^{\frac{20-5k-60s_{1}}{20k}}\langle D\rangle^{s} z_{m}\right\|_{L_{x}^{\frac{20k}{5k-4+12s_{1}}}L_{t}^{\frac{5k}{2-6s_{1}}}}\right)\nonumber\\
&\quad\times\left(\prod_{n=j+1}^{k}N_{n}^{\frac{5k-20+20s_{1}}{20k}-\frac{20-5k-60s_{1}}{20k}}\left\|v_{n}\right\|_{X^{0,b}}\right)
 N_{k+1}^{-s_{1}}\|v_{k+1}\|_{X^{\sigma,b}},\label{3.014}
\end{align}
where we use the following facts
\begin{eqnarray*}
&&\frac{2(1-3s_{1})}{5}+\frac{1+12s_{1}}{10}=\frac{1}{2},\quad \frac{1+12s_{1}}{5}-\frac{1-3s_{1}}{5}=3s_{1},\\
&&\frac{4(1-3s_{1})}{5k}-\frac{5k-4+12s_{1}}{20k}=\frac{20-5k-60s_{1}}{20k},\\
&&\frac{5k-20+20s_{1}}{20k}-\frac{20-5k-60s_{1}}{20k}=\frac{1}{2}-\frac{2}{k}+\frac{4s_{1}}{k}\in(0,\sigma).
\end{eqnarray*}
By using \eqref{3.014} and H\"{o}lder inequality, then, we have
\begin{align}
\sum\left\|D^{2s_{1}}\langle D\rangle^{\sigma}\left(\prod_{m=1}^{j}z_{m}\prod_{n=j+1}^{k+1}v_{n}\right)\right\|_{L_{x}^{\frac{1}{1-s_{1}}}L_{t}^{2}}
&\leq C\|z\|_{Y(\tiny\bf{R})}^{j}\|v\|_{X^{\sigma,b}}^{k+1-j}\nonumber\\
&\leq C\left(\|z\|_{Y(\tiny\bf{R})}^{k+1}+\|v\|_{X^{\sigma,b}}^{k+1}\right).\label{3.015}
\end{align}

When $N_{j}\leq N_{k+1}\leq 1$,
by using\eqref{2.041} and H\"{o}lder inequality, Lemmas 2.11, 2.12 and a proof similar to \eqref{3.05}, we have
\begin{align}
&\quad\left\|D^{2s_{1}}\langle D\rangle^{\sigma}\left(\prod_{m=1}^{j}z_{m}\prod_{n=j+1}^{k+1}v_{n}\right)\right\|_{L_{x}^{\frac{1}{1-s_{1}}}L_{t}^{2}}
\leq CN_{k+1}^{2s_{1}}\left\|\left(\prod_{m=1}^{j}z_{m}\prod_{n=j+1}^{k+1}v_{n}\right)\right\|_{L_{x}^{\frac{1}{1-s_{1}}}L_{t}^{2}}\nonumber\\
&\leq
C\left(\prod_{m=1}^{j}\|z_{m}\|_{L_{x}^{\frac{5k}{4-3s_{1}}}L_{t}^{\frac{5k}{2-4s_{1}}}}\right)
\left(\prod_{n=j+1}^{k}\|v_{n}\|_{L_{x}^{\frac{5k}{4-3s_{1}}}L_{t}^{\frac{5k}{2-4s_{1}}}}\right)N_{k+1}^{2s_{1}}
\left\|v_{k+1}\right\|_{L_{x}^{\frac{5}{1-2s_{1}}}L_{t}^{\frac{10}{1+8s_{1}}}}\nonumber\\
&\leq
C\left(\prod_{m=1}^{j}\|z_{m}\|_{L_{x}^{\frac{5k}{4-3s_{1}}}L_{t}^{\frac{5k}{2-4s_{1}}}}\right)
\nonumber\\
&\quad\times\left(\prod_{n=j+1}^{k}\|v_{n}\|_{L_{x}^{\frac{5k}{4-3s_{1}}}L_{t}^{\frac{5k}{2-4s_{1}}}}\right)\left\|D^{2s_{1}}\langle D\rangle^{\sigma}v_{k+1}\right\|_{L_{x}^{\frac{5}{1-2s_{1}}}L_{t}^{\frac{10}{1+8s_{1}}}}\nonumber\\
&\leq C\left(\prod_{m=1}^{j}N_{m}^{\frac{1}{2}-\frac{2}{k}+\frac{3s_{1}}{k}}
\left\|D^{\frac{4-k-8s_{1}}{4k}}\langle D\rangle^{s}z_{m}\right\|_{L_{x}^{\frac{20k}{5k-4(1-2s_{1})}}L_{t}^{\frac{5k}{2-4s_{1}}}}\right)\nonumber\\
&\quad\times\left(\prod_{n=j+1}^{k}N_{n}^{\frac{1}{2}-\frac{2}{k}+\frac{3s_{1}}{k}}
\|v_{n}\|_{X^{0,b}}\right)\|v_{k+1}\|_{X^{\sigma,b}}.\label{3.016}
\end{align}
By using \eqref{3.016} and H\"{o}lder inequality, we have
\begin{align}
\sum\left\|D^{2s_{1}}\langle D\rangle^{\sigma}\left(\prod_{m=1}^{j}z_{m}\prod_{n=j+1}^{k+1}v_{n}\right)\right\|_{L_{x}^{\frac{1}{1-s_{1}}}L_{t}^{2}}
&\leq C\|z\|_{Y(\tiny\bf{R})}^{j}\|v\|_{X^{\sigma,b}}^{k+1-j}\nonumber\\
&\leq C\left(\|z\|_{Y(\tiny\bf{R})}^{k+1}+\|v\|_{X^{\sigma,b}}^{k+1}\right).\label{3.017}
\end{align}

When $N_{k+1}\geq N_{j}> 1$, there exists a $\tilde{i}$, such that $N_{\tilde{i}}\leq1$ and $N_{j}\geq N_{\tilde{i}+1}>1$. By using  \eqref{2.041} and H\"{o}lder inequality, Lemmas 2.11, 2.12
as well as a proof similar to \eqref{3.014}, we have
\begin{align}
&\quad\left\|D^{2s_{1}}\langle D\rangle^{\sigma}\left(\prod_{m=1}^{j}z_{m}\prod_{i_{2}=n}^{k+1}v_{n}\right)\right\|_{L_{x}^{\frac{1}{1-s_{1}}}L_{t}^{2}}
\leq N_{k+1}^{2s_{1}}N_{k+1}^{\sigma}
\left\|\left(\prod_{m=1}^{j}z_{m}\prod_{n=j+1}^{k+1}v_{n}\right)\right\|_{L_{x}^{\frac{1}{1-s_{1}}}L_{t}^{2}}\nonumber\\
&\leq
C\left(\prod_{m=1}^{\tilde{i}}\|z_{m}\|_{L_{x}^{\frac{5k}{4-2s_{1}}}L_{t}^{\frac{5k}{2-6s_{1}}}}\right)
\nonumber\\
&\quad\times\left(\prod_{m=\tilde{i}+1}^{j}\|z_{m}\|_{L_{x}^{\frac{5k}{4-2s_{1}}}L_{t}^{\frac{5k}{2-6s_{1}}}}\right)
\left(\prod_{n=j+1}^{k}\|v_{n}\|_{L_{x}^{\frac{5k}{4-2s_{1}}}L_{t}^{\frac{5k}{2-6s_{1}}}}\right)
N_{k+1}^{2s_{1}}N_{k+1}^{\sigma}\left\|v_{k+1}\right\|_{L_{x}^{\frac{5}{1-3s_{1}}}L_{t}^{\frac{10}{1+12s_{1}}}}\nonumber\\
&\leq
C\left(\prod_{m=1}^{\tilde{i}}\|z_{m}\|_{L_{x}^{\frac{5k}{4-2s_{1}}}L_{t}^{\frac{5k}{2-6s_{1}}}}\right)
\nonumber\\
&\quad\times\left(\prod_{m=\tilde{i}+1}^{j}\|z_{m}\|_{L_{x}^{\frac{5k}{4-2s_{1}}}L_{t}^{\frac{5k}{2-6s_{1}}}}\right)
\left(\prod_{n=j+1}^{k}\|v_{n}\|_{L_{x}^{\frac{5k}{4-2s_{1}}}L_{t}^{\frac{5k}{2-6s_{1}}}}\right)
\left\|D^{2s_{1}}\langle D\rangle^{\sigma}v_{k+1}\right\|_{L_{x}^{\frac{5}{1-3s_{1}}}L_{t}^{\frac{10}{1+12s_{1}}}}\nonumber\\
&\leq C\left(\prod_{m=1}^{\tilde{i}}N_{m}^{\frac{1}{2}-\frac{2}{k}+\frac{4s_{1}}{k}}\left\|D^{\frac{20-5k-60s_{1}}{20k}}\langle D\rangle^{s} z_{m}\right\|_{L_{x}^{\frac{20k}{5k-4+12s_{1}}}L_{t}^{\frac{5k}{2-6s_{1}}}}\right)\nonumber\\
&\quad\times\left(\prod_{n=j+1}^{k}N_{n}^{\frac{1}{2}-\frac{2}{k}+\frac{4s_{1}}{k}}\|v_{n}\|_{X^{0,b}}\right)
N_{k+1}^{-s_{1}}\|v_{k+1}\|_{X^{\sigma,b}}\nonumber\\
&\quad\times\prod_{m=\tilde{i}+1}^{j}
\left\|D^{-\frac{2s_{1}}{k}+\frac{(k-4+4s_{1})a}{2k}}\langle D\rangle^{s}\langle D\rangle^{-s} D^{\frac{2s_{1}}{k}-\frac{(k-4+4s_{1})a}{2k}}z_{m}\right\|
_{L_{x}^{\frac{5k}{4-2s_{1}}}L_{t}^{\frac{5k}{2-6s_{1}}}}\nonumber\\
&\leq C\left(\prod_{m=1}^{\tilde{i}}N_{m}^{\frac{1}{2}-\frac{2}{k}+\frac{4s_{1}}{k}}\left\|D^{\frac{20-5k-60s_{1}}{20k}}\langle D\rangle^{s} z_{m}\right\|_{L_{x}^{\frac{20k}{5k-4+12s_{1}}}L_{t}^{\frac{5k}{2-6s_{1}}}}\right)\nonumber\\
&\quad\times\left(\prod_{n=j+1}^{k}N_{n}^{\frac{1}{2}-\frac{2}{k}+\frac{4s_{1}}{k}}\|v_{n}\|_{X^{0,b}}\right)
N_{k+1}^{-s_{1}}\|v_{k+1}\|_{X^{\sigma,b}}\nonumber\\
&\quad\times\prod_{m=\tilde{i}+1}^{j}N_{m}^{-s+\frac{2s_{1}}{k}-\frac{(k-4+4s_{1})a}{2k}}
\left\| D^{-\frac{2s_{1}}{k}+\frac{(k-4+4s_{1})a}{2k}}\langle D\rangle^{s}z_{m}\right\|
_{L_{x}^{\frac{5k}{4-2s_{1}}}L_{t}^{\frac{5k}{2-6s_{1}}}}.\label{3.018}
\end{align}
By using \eqref{3.018},  H\"{o}lder inequality and $a\in\mathbf{N}$ with
\begin{eqnarray*}
&&a> -\frac{2(ks-2s_{1})}{k-4+4s_{1}},
\end{eqnarray*}
it follows that the powers of $N_{m},\,m\geq \tilde{i}+1$ in \eqref{3.018} are negative. Since $N_{m}\geq1,\,m\geq \tilde{i}+1$, we have
\begin{align}
\sum\left\|D^{2s_{1}}\langle D\rangle^{\sigma}\left(\prod_{m=1}^{j}z_{m}\prod_{n=j+1}^{k+1}v_{n}\right)\right\|_{L_{x}^{\frac{1}{1-s_{1}}}L_{t}^{2}}
&\leq C\|z\|_{Y(\tiny\bf{R})}^{j}\|v\|_{X^{\sigma,b}}^{k+1-j}\nonumber\\
&\leq C\left(\|z\|_{Y(\tiny\bf{R})}^{k+1}+\|v\|_{X^{\sigma,b}}^{k+1}\right).\label{3.019}
\end{align}

\noindent{\bf Case 3B:} $N_{j}> N_{k+1}$, we consider $N_{j}>1$ and $N_{j}\leq1$, respectively.

When $N_{j}\leq 1$,
by using \eqref{2.041} and H\"{o}lder inequality, Lemmas 2.11, 2.12 and a proof similar to \eqref{3.05}, we have
\begin{align}
&\quad\left\|D^{2s_{1}}\langle D\rangle^{\sigma}\left(\prod_{m=1}^{j}z_{m}\prod_{n=j+1}^{k+1}v_{n}\right)\right\|_{L_{x}^{\frac{1}{1-s_{1}}}L_{t}^{2}}
\leq N_{j}^{2s_{1}}\left\|\left(\prod_{m=1}^{j}z_{m}\prod_{n=j+1}^{k+1}v_{n}\right)\right\|_{L_{x}^{\frac{1}{1-s_{1}}}L_{t}^{2}}\nonumber\\
&\leq
C\left(\prod_{m=1}^{j-1}\|z_{m}\|_{L_{x}^{\frac{5k}{4-3s_{1}}}L_{t}^{\frac{5k}{2-4s_{1}}}}\right)
\left(\prod_{n=j+1}^{k+1}\|v_{n}\|_{L_{x}^{\frac{5k}{4-3s_{1}}}L_{t}^{\frac{5k}{2-4s_{1}}}}\right)
N_{j}^{2s_{1}}\left\|z_{j}\right\|_{L_{x}^{\frac{5}{1-2s_{1}}}L_{t}^{\frac{10}{1+8s_{1}}}}\nonumber\\
&\leq
C\left(\prod_{m=1}^{j-1}\|z_{m}\|_{L_{x}^{\frac{5k}{4-3s_{1}}}L_{t}^{\frac{5k}{2-4s_{1}}}}\right)
\nonumber\\
&\quad\times\left(\prod_{n=j+1}^{k+1}\|v_{n}\|_{L_{x}^{\frac{5k}{4-3s_{1}}}L_{t}^{\frac{5k}{2-4s_{1}}}}\right)
\left\|D^{2s_{1}}\langle D\rangle^{s}z_{j}\right\|_{L_{x}^{\frac{5}{1-2s_{1}}}L_{t}^{\frac{10}{1+8s_{1}}}}\nonumber\\
&\leq C\left(\prod_{m=1}^{j-1}N_{m}^{\frac{1}{2}-\frac{2}{k}+\frac{3s_{1}}{5k}}\left\|D^{\frac{4-k-8s_{1}}{4k}}\langle D\rangle^{s}z_{m}\right\|
_{L_{x}^{\frac{20k}{5k-4(1-2s_{1})}}L_{t}^{\frac{5k}{2-4s_{1}}}}\right)
\left(\prod_{n=j+1}^{k+1}N_{n}^{\frac{1}{2}-\frac{2}{k}+\frac{3s_{1}}{5k}}\|v_{n}\|_{X^{0,b}}\right)\nonumber\\
&\quad\times\left\|D^{2s_{1}}\langle D\rangle^{s} z_{j}\right\|_{L_{x}^{\frac{5}{1-2s_{1}}}L_{t}^{\frac{10}{1+8s_{1}}}}.\label{3.020}
\end{align}
By using \eqref{3.020} and H\"{o}lder inequality,, we have
\begin{align}
\sum\left\|D^{2s_{1}}\langle D\rangle^{\sigma}\left(\prod_{m=1}^{j}z_{m}\prod_{n=j+1}^{k+1}v_{n}\right)\right\|_{L_{x}^{\frac{1}{1-s_{1}}}L_{t}^{2}}
&\leq C\|z\|_{Y(\tiny\bf{R})}^{j}\|v\|_{X^{\sigma,b}}^{k+1-j}\nonumber\\
&\leq C\left(\|z\|_{Y(\tiny\bf{R})}^{k+1}+\|v\|_{X^{\sigma,b}}^{k+1}\right).\label{3.021}
\end{align}

When $N_{j}> 1$, we consider $j=1$ and $j\geq2$, respectively.

For $j=1$, that is $N_{1}>1$, by using \eqref{2.041} and H\"{o}lder inequality, Lemmas 2.11, 2.12 as well as a proof similar to \eqref{3.012}, we have
\begin{align}
&\quad\left\|D^{2s_{1}}\langle D\rangle^{\sigma}\left(z_{1}\prod_{n=2}^{k+1}v_{n}\right)\right\|_{L_{x}^{\frac{1}{1-s_{1}}}L_{t}^{2}}
\leq N_{1}^{2s_{1}}N_{1}^{\sigma}\left\|\left(z_{1}\prod_{n=2}^{k+1}v_{n}\right)\right\|_{L_{x}^{\frac{1}{1-s_{1}}}L_{t}^{2}}\nonumber\\
&\leq
C\left(\prod_{n=2}^{k+1}\|v_{n}\|_{L_{x}^{\frac{5k}{4(1+s_{1})}}L_{t}^{\frac{5k}{2(1+s_{1})}}}\right)
N_{1}^{2s_{1}}N_{1}^{\sigma}\left\|z_{1}\right\|_{L_{x}^{\frac{5}{1-9s_{1}}}L_{t}^{\frac{10}{1-4s_{1}}}}\nonumber\\
&\leq
C\left(\prod_{n=2}^{k+1}\|v_{n}\|_{L_{x}^{\frac{5k}{4(1+s_{1})}}L_{t}^{\frac{5k}{2(1+s_{1})}}}\right)
\left\|D^{2s_{1}}\langle D\rangle^{\sigma}z_{1}\right\|_{L_{x}^{\frac{5}{1-9s_{1}}}L_{t}^{\frac{10}{1-4s_{1}}}}\nonumber\\
&\leq C\left(\prod_{n=2}^{k+1}\|v_{n}\|_{L_{x}^{\frac{5k}{4(1+s_{1})}}L_{t}^{\frac{5k}{2(1+s_{1})}}}\right)\nonumber\\
&\quad\times \left\|D^{(4a+1)s_{1}}\langle D\rangle^{s}\langle D\rangle^{-s} D^{-(4a+1)s_{1}}D^{2s_{1}}\langle D\rangle^{\sigma}z_{1}\right\|_{L_{x}^{\frac{5}{1-9s_{1}}}L_{t}^{\frac{10}{1-4s_{1}}}}\nonumber\\
&\leq C\left(\prod_{n=2}^{k+1}N_{n}^{\frac{1}{2}-\frac{2}{k}-\frac{2s_{1}}{k}}\|v_{n}\|_{X^{0,b}}\right)\nonumber\\
&\quad\times N_{1}^{2s_{1}+\sigma-s-(4a+1)s_{1}}\left\|D^{(4a+1)s_{1}}\langle D\rangle^{s}z_{1}\right\|_{L_{x}^{\frac{5}{1-9s_{1}}}L_{t}^{\frac{10}{1-4s_{1}}}}
.\label{3.022}
\end{align}
By using \eqref{3.022}, H\"{o}lder inequality, and $a\in\mathbf{N}$ with
\begin{eqnarray*}
&&a>\frac{\sigma-s}{4s_{1}}+\frac{1}{4},
\end{eqnarray*}
it follows that the power of $N_{1}$ in \eqref{3.022} is negative. Since $N_{1}>1$, we have
\begin{align}
\sum\left\|D^{2s_{1}}\langle D\rangle^{\sigma}\left(z_{1}\prod_{n=j+1}^{k+1}v_{n}\right)\right\|_{L_{x}^{\frac{1}{1-s_{1}}}L_{t}^{2}}
&\leq C\|z\|_{Y(\tiny\bf{R})}\|v\|_{X^{\sigma,b}}^{k}\nonumber\\
&\leq C\left(\|z\|_{Y(\tiny\bf{R})}^{k+1}+\|v\|_{X^{\sigma,b}}^{k+1}\right).\label{3.023}
\end{align}
For $j\geq2$, $N_{j}>1$,  there exists a $\tilde{j}$, such that $N_{\tilde{j}}\leq1$ and $N_{\tilde{j}+1}>1$. By using \eqref{2.041} and H\"{o}lder inequality, Lemmas 2.11, 2.12 as well as a proof similar to \eqref{3.012}, we have
\begin{align}
&\quad\left\|D^{2s_{1}}\langle D\rangle^{\sigma}\left(\prod_{m=1}^{j}z_{m}\prod_{n=j+1}^{k+1}v_{n}\right)\right\|_{L_{x}^{\frac{1}{1-s_{1}}}L_{t}^{2}}
\leq N_{j}^{2s_{1}}N_{j}^{\sigma}
\left\|\left(\prod_{m=1}^{j}z_{m}\prod_{n=j+1}^{k+1}v_{n}\right)\right\|_{L_{x}^{\frac{1}{1-s_{1}}}L_{t}^{2}}\nonumber\\
&\leq
C\left(\prod_{m=1}^{\tilde{j}}\|z_{m}\|_{L_{x}^{\frac{5k}{4(1+s_{1})}}L_{t}^{\frac{5k}{2(1+s_{1})}}}\right)
\left(\prod_{m=\tilde{j}+1}^{j-1}\|z_{m}\|_{L_{x}^{\frac{5k}{4(1+s_{1})}}L_{t}^{\frac{5k}{2(1+s_{1})}}}\right)\nonumber\\
&\quad\times\left(\prod_{n=j+1}^{k+1}\|v_{n}\|_{L_{x}^{\frac{5k}{4(1+s_{1})}}L_{t}^{\frac{5k}{2(1+s_{1})}}}\right)
N_{j}^{2s_{1}}N_{j}^{\sigma}\left\|z_{j}\right\|_{L_{x}^{\frac{5}{1-9s_{1}}}L_{t}^{\frac{10}{1-4s_{1}}}}\nonumber\\
&\leq
C\left(\prod_{m=1}^{\tilde{j}}\|z_{m}\|_{L_{x}^{\frac{5k}{4(1+s_{1})}}L_{t}^{\frac{5k}{2(1+s_{1})}}}\right)
\left(\prod_{m=\tilde{j}+1}^{j-1}\|z_{m}\|_{L_{x}^{\frac{5k}{4(1+s_{1})}}L_{t}^{\frac{5k}{2(1+s_{1})}}}\right)\nonumber\\
&\quad\times\left(\prod_{n=j+1}^{k+1}\|v_{n}\|_{L_{x}^{\frac{5k}{4(1+s_{1})}}L_{t}^{\frac{5k}{2(1+s_{1})}}}\right)
\left\|D^{2s_{1}}\langle D\rangle^{\sigma}z_{j}\right\|_{L_{x}^{\frac{5}{1-9s_{1}}}L_{t}^{\frac{10}{1-4s_{1}}}}\nonumber\\
&\leq C\left(\prod_{m=1}^{\tilde{j}}N_{m}^{\frac{1}{2}-\frac{2}{k}-\frac{2s_{1}}{k}}\left\|D^{\frac{4(1+s_{1})-k}{4k}}\langle D\rangle^{s} z_{m}\right\|_{L_{x}^{\frac{20k}{5k-4(1+s_{1})}}L_{t}^{\frac{5k}{2(1+s_{1})}}}\right)
\left(\prod_{n=j+1}^{k+1}N_{n}^{\frac{1}{2}-\frac{2}{k}-\frac{2s_{1}}{k}}\|v_{n}\|_{X^{0,b}}\right)\nonumber\\
&\quad\times\left(\prod_{m=\tilde{j}+1}^{j-1}\left\|D^{\frac{a(k-4(1+s_{1}))}{2k}}\langle D\rangle^{s} \langle D\rangle^{-s}D^{-\frac{a(k-4(1+s_{1}))}{2k}} z_{m}\right\|_{L_{x}^{\frac{5k}{4(1+s_{1})}}L_{t}^{\frac{5k}{2(1+s_{1})}}}\right)\nonumber\\
&\quad\times
\left\|D^{(4a+1)s_{1}}\langle D\rangle^{s}\langle D\rangle^{-s}D^{-(4a+1)s_{1}}D^{2s_{1}}\langle D\rangle^{\sigma}z_{j}\right\|_{L_{x}^{\frac{5}{1-9s_{1}}}L_{t}^{\frac{10}{1-4s_{1}}}}\nonumber\\
&\leq C\left(\prod_{m=1}^{\tilde{j}}N_{m}^{\frac{1}{2}-\frac{2}{k}-\frac{2s_{1}}{k}}\left\|D^{\frac{4(1+s_{1})-k}{4k}}\langle D\rangle^{s} z_{m}\right\|_{L_{x}^{\frac{20k}{5k-4(1+s_{1})}}L_{t}^{\frac{5k}{2(1+s_{1})}}}\right)
\left(\prod_{n=j+1}^{k+1}N_{n}^{\frac{1}{2}-\frac{2}{k}-\frac{2s_{1}}{k}}\|v_{n}\|_{X^{0,b}}\right)\nonumber\\
&\quad\times\left(\prod_{m=\tilde{j}+1}^{j-1}N_{m}^{-s-\frac{a(k-4(1+s_{1}))}{2k}}\left\|D^{\frac{a(k-4(1+s_{1}))}{2k}}\langle D\rangle^{s} z_{m}\right\|_{L_{x}^{\frac{5k}{4(1+s_{1})}}L_{t}^{\frac{5k}{2(1+s_{1})}}}\right)\nonumber\\
&\quad\times N_{j}^{2s_{1}+\sigma-s-(4a+1)s_{1}}\left\|D^{(4a+1)s_{1}}\langle D\rangle^{s}z_{j}\right\|_{L_{x}^{\frac{5}{1-9s_{1}}}L_{t}^{\frac{10}{1-4s_{1}}}}.\label{3.024}
\end{align}
By using \eqref{3.024}, H\"{o}lder inequality and $a\in\mathbf{N}$ with
\begin{eqnarray*}
&&a>\max\left\{ -\frac{2ks}{k-4-4s_{1}},\,\frac{\sigma-s}{4s_{1}}+\frac{1}{4}\right\},
\end{eqnarray*}
it follows that the powers of $N_{m},\,m\geq\tilde{j}$ in \eqref{3.024} are negative. Since $N_{m}>1,\,m\geq\tilde{j}$, we have
\begin{align}
\sum\left\|D^{2s_{1}}\langle D\rangle^{\sigma}\left(\prod_{m=1}^{j}z_{m}\prod_{n=j+1}^{k+1}v_{n}\right)\right\|_{L_{x}^{\frac{1}{1-s_{1}}}L_{t}^{2}}
&\leq C\|z\|_{Y(\tiny\bf{R})}^{j}\|v\|_{X^{\sigma,b}}^{k+1-j}\nonumber\\
&\leq C\left(\|z\|_{Y(\tiny\bf{R})}^{k+1}+\|v\|_{X^{\sigma,b}}^{k+1}\right).\label{3.025}
\end{align}

The proof of Proposition 3.1 is completed.

\begin{Lemma}\label{lem3.1}
Let $f(x)\in H^{s}(\R)$. Its randomization $f^{\omega}(x)$ is defined by \eqref{1.03}. For any $\lambda>0$, there exist constants $C_{1}, c>0$,  such that
\begin{eqnarray}
&&\mathbb{P}\left(\left\{\omega: \left\|U(t)f^{\omega}\right\|_{Y(\tiny\bf{R})} > \lambda\right\}\right) \leq C_{1}\exp\left(-c\frac{\lambda^{2}}{\|f\|_{H^{s}}^{2}}\right).\label{3.026}
\end{eqnarray}
In particular, for almost every $\omega\in\Omega$, we have
\begin{eqnarray*}
&&\left\|U(t)f^{\omega}\right\|_{Y(\tiny\bf{R})}<\infty.
\end{eqnarray*}

\end{Lemma}
\noindent{\bf Proof.}
We take $r\geq2$.
Noting that
\begin{eqnarray*}
&&\left\|U(t)f^{\omega}\right\|_{Y(\SR)}=\left\|\left\|U(t)f^{\omega}\right\|_{Y_{N}(\SR)}\right\|_{\ell_{N}^{2}},\\
&&\left\|U(t)f^{\omega}\right\|_{Y_{N}(\SR)}=\left\|D^{\frac{4-k-8s_{1}}{4k}}\langle D\rangle^{s}P_{N}U(t)f^{\omega}\right\|_{L_{x}^{\frac{20k}{5k-4(1-2s_{1})}}L_{t}^{\frac{5k}{2-4s_{1}}}}\\
&&\quad+\left\|D^{2s_{1}}\langle D\rangle^{s} P_{N}U(t)f^{\omega}\right\|_{L_{x}^{\frac{5}{1-2s_{1}}}L_{t}^{\frac{10}{1+8s_{1}}}}\\
&&\quad+\left\|D^{\frac{4(1+s_{1})-k}{4k}}\langle D\rangle^{s}P_{N}U(t)f^{\omega}\right\|_{L_{x}^{\frac{20k}{5k-4(1+s_{1})}}L_{t}^{\frac{5k}{2(1+s_{1})}}}\\
&&\quad+\left\|D^{\frac{2s_{1}+a(k-4(1-s_{1}))}{2(k+1)}} \langle D\rangle^{s}P_{N} U(t)f^{\omega}\right\|_{L_{x}^{\frac{k+1}{1-s_{1}}}L_{t}^{2(k+1)}}\\
&&\quad+\left\|D^{\frac{a(k-4(1+s_{1}))}{2k}}\langle D\rangle^{s}P_{N}U(t)f^{\omega}\right\|_{L_{x}^{\frac{5k}{4(1+s_{1})}}L_{t}^{\frac{5k}{2(1+s_{1})}}}\\
&&\quad+\left\|D^{\frac{4-k-12s_{1}}{4k}}\langle D\rangle^{s}P_{N}U(t)f^{\omega}\right\|_{L_{x}^{\frac{20k}{5k-4+12s_{1}}}L_{t}^{\frac{5k}{2-6s_{1}}}}\\
&&\quad+\left\|D^{(4a+1)s_{1}}\langle D\rangle^{s}P_{N}U(t)f^{\omega}\right\|_{L_{x}^{\frac{5}{1-9s_{1}}}L_{t}^{\frac{10}{1-4s_{1}}}}\\
&&\quad+\left\| D^{-\frac{2s_{1}}{k}+\frac{(k-4+4s_{1})a}{2k}}\langle D\rangle^{s}P_{N}U(t)f^{\omega}\right\|
_{L_{x}^{\frac{5k}{4-2s_{1}}}L_{t}^{\frac{5k}{2-6s_{1}}}}\\
&&=\sum_{i=1}^{8}I_{i}.
\end{eqnarray*}
For $I_{1}$, by using Lemma 2.4, Minkowski inequality and \eqref{2.03}, we have
\begin{eqnarray}
&&\|I_{1}\|_{L_{\omega}^{r}}=\left\|\left\|D^{\frac{4-k-8s_{1}}{4k}}\langle D\rangle^{s}P_{N}U(t)f^{\omega}\right\|_{L_{x}^{\frac{20k}{5k-4(1-2s_{1})}}L_{t}^{\frac{5k}{2-4s_{1}}}}\right\|_{L_{\omega}^{r}}\nonumber\\
&&\leq C\left\|\left\|\langle D\rangle^{s}P_{N}f^{\omega}\right\|_{L_{x}^{2}}\right\|_{L_{\omega}^{r}}\leq C\left\|\left\|\langle D\rangle^{s}P_{N}f^{\omega}\right\|_{L_{\omega}^{r}}\right\|_{L_{x}^{2}}\nonumber\\
&&\leq C\sqrt{r}\left\|\left\|\langle D\rangle^{s}P_{N}\psi_{k}(D)f\right\|_{\ell_{k}^{2}}\right\|_{L_{x}^{2}}\nonumber\\
&&\leq C\sqrt{r}\left\|\langle D\rangle^{s}P_{N}f\right\|_{L_{x}^{2}}.\label{3.027}
\end{eqnarray}
For $I_{i}$, $i=2, 3, 6$,  using \eqref{2.03}, Lemma 2.4, and a proof similar to that of \eqref{3.027}, we obtain
\begin{eqnarray}
&&\|I_{i}\|_{L_{\omega}^{r}}\leq C\sqrt{r}\left\|\langle D\rangle^{s}P_{N}f\right\|_{L_{x}^{2}}.\label{3.028}
\end{eqnarray}
For $I_{4}$, noting that
\begin{eqnarray*}
&&\frac{k+1}{1-s_{1}}\left(1-\frac{2}{2(k+1)}\right)=\frac{k}{1-s_{1}}>4,\\
&&\frac{2}{2(k+1)}-\frac{1-s_{1}}{k+1}+\frac{a(1-s_{1})}{2(k+1)}\left(\frac{k+1}{1-s_{1}}\left(1-\frac{2}{2(k+1)}\right)-4\right)\\
&&=\frac{2s_{1}+a(k-4(1-s_{1}))}{2(k+1)},
\end{eqnarray*}
by using Proposition 2.1, we have
\begin{eqnarray}
&&\|I_{4}\|_{L_{\omega}^{r}}\leq C\sqrt{r}\left\|\langle D\rangle^{s}P_{N}f\right\|_{L_{x}^{2}}.\label{3.029}
\end{eqnarray}
For $I_{5}$, noting that
\begin{eqnarray*}
&&\frac{5k}{4(1+s_{1})}\left(1-\frac{4(1+s_{1})}{5k}\right)=\frac{5k-4(1+s_{1})}{4(1+s_{1})}>4,\\
&&\frac{4(1+s_{1})}{5k}-\frac{4(1+s_{1})}{5k}+\frac{2a(1+s_{1})}{5k}\left(\frac{5k}{4(1+s_{1})}\left(1-\frac{4(1+s_{1})}{5k}\right)-4\right)\\
&&=\frac{a(k-4(1+s_{1}))}{2k},
\end{eqnarray*}
by using Proposition 2.1, we have
\begin{eqnarray}
&&\|I_{5}\|_{L_{\omega}^{r}}\leq C\sqrt{r}\left\|\langle D\rangle^{s}P_{N}f\right\|_{L_{x}^{2}}.\label{3.030}
\end{eqnarray}
For $I_{7}$, noting that
\begin{eqnarray*}
&&\frac{5}{1-9s_{1}}\left(1-\frac{2(1-4s_{1})}{10}\right)=\frac{4(1+s_{1})}{1-9s_{1}}>4,\\
&&\frac{2(1-4s_{1})}{10}-\frac{1-9s_{1}}{5}+\frac{a(1-9s_{1})}{10}\left(\frac{5}{1-9s_{1}}\left(1-\frac{2(1-4s_{1})}{10}\right)-4\right)\\
&&=(4a+1)s_{1},
\end{eqnarray*}
by using Proposition 2.1, we have
\begin{eqnarray}
&&\|I_{7}\|_{L_{\omega}^{r}}\leq C\sqrt{r}\left\|\langle D\rangle^{s}P_{N}f\right\|_{L_{x}^{2}}.\label{3.031}
\end{eqnarray}
For $I_{8}$, noting that
\begin{eqnarray*}
&&\frac{5k}{4-2s_{1}}\left(1-\frac{2(2-6s_{1})}{5k}\right)=\frac{5k-4+12s_{1}}{4-2s_{1}}>4,\\
&&\frac{2(2-6s_{1})}{5k}-\frac{4-2s_{1}}{5k}+\frac{(2-s_{1})a}{5k}\left(\frac{5k}{4-2s_{1}}\left(1-\frac{2(2-6s_{1})}{5k}\right)-4\right)\\
&&=-\frac{2s_{1}}{k}+\frac{(k-4+4s_{1})a}{2k},
\end{eqnarray*}
by using Proposition 2.1, we have
\begin{eqnarray}
&&\|I_{8}\|_{L_{\omega}^{r}}\leq C\sqrt{r}\left\|\langle D\rangle^{s}P_{N}f\right\|_{L_{x}^{2}}.\label{3.032}
\end{eqnarray}
Finally, combining the estimates from \eqref{3.027} to \eqref{3.032}, we obtain
\begin{eqnarray}
&&\left\|\left\|U(t)f^{\omega}\right\|_{Y_{N}(\SR)}\right\|_{L_{\omega}^{r}} \leq \sum_{i=1}^{8}\|I_{i}\|_{L_{\omega}^{r}}\nonumber\\
&&\leq C\sqrt{r}\left\|\langle D\rangle^{s}P_{N}f\right\|_{L_{x}^{2}}.\label{3.033}
\end{eqnarray}
By using Minkowski inequality and \eqref{3.033}, we have
\begin{eqnarray}
&&\left\|\left\|U(t)f^{\omega}\right\|_{Y(\SR)}\right\|_{L_{\omega}^{r}}=\left\|\left\|\left\|U(t)f^{\omega}\right\|_{Y_{N}(\SR)}\right\|
_{\ell_{N}^{2}}\right\|_{L_{\omega}^{r}}\nonumber\\
&&\leq C\left\|\left\|\left\|U(t)f^{\omega}\right\|_{Y_{N}(\SR)}\right\|_{L_{\omega}^{r}}
\right\|_{\ell_{N}^{2}}\nonumber\\
&&\leq C\sqrt{r}\left\|\left\|\langle D\rangle^{s}P_{N}f\right\|_{L_{x}^{2}}\right\|_{\ell_{N}^{2}}\nonumber\\
&&\leq C\sqrt{r}\left\|f\right\|_{H^{s}}.\label{3.034}
\end{eqnarray}
By using \eqref{3.034} and Lemma 2.5, we have that \eqref{3.026} is valid.

\noindent We denote
\begin{eqnarray}
&&\Omega=\left\{\omega: \left\|U(t)f^{\omega}\right\|_{Y(\tiny\bf{R})}<\infty\right\}.\label{3.035}
\end{eqnarray}
From \eqref{3.035}, we have
\begin{eqnarray*}
&&\Omega=\bigcup_{n=1}^{\infty}\left\{\omega: \left\|U(t)f^{\omega}\right\|_{Y(\tiny\bf{R})}\leq n\right\}.
\end{eqnarray*}
Its complement is given by
\begin{eqnarray}
&&\Omega^{c}=\bigcap_{n=1}^{\infty}\left\{\omega: \left\|U(t)f^{\omega}\right\|_{Y(\tiny\bf{R})}> n\right\}.\label{3.036}
\end{eqnarray}
Now, combining \eqref{3.036} with  \eqref{3.026}, we obtain
\begin{eqnarray*}
&&0\leq \mathbb{P}\left(\Omega^{c}\right)=\lim_{n\rightarrow\infty}\mathbb{P}\left(\left\{\omega: \left\|U(t)f^{\omega}\right\|_{Y(\tiny\bf{R})}> n\right\}\right)\nonumber\\
&&\leq C_{1}\lim_{n\rightarrow\infty}\exp\left(-c\frac{n^{2}}{\|f\|_{H^{s}}^{2}}\right)=0.
\end{eqnarray*}
Therefore, it follows that
\begin{eqnarray}
&&\mathbb{P}(\Omega)=1-\mathbb{P}\left(\Omega^{c}\right)=1.\label{3.037}
\end{eqnarray}

The proof of Lemma 3.1 is completed.

\bigskip

\bigskip

\section{Proof of Theorem  1.1}

\setcounter{equation}{0}

 \setcounter{Theorem}{0}

\setcounter{Lemma}{0}

\setcounter{Proposition}{0}

\setcounter{section}{4}

Before proving Theorem 1.1, we establish local well-posedness for
\begin{eqnarray}
&&v_{t}+v_{xxx}+\frac{1}{k+1}\partial_{x}[(v+G)^{k+1}]=0,\label{4.01}\\
&&v(0,x)=v_{0}(x),\label{4.02}
\end{eqnarray}
which plays an important role in proving of Theorem 1.1.
\begin{Theorem}\label{Theorem 4.1}(Local well-posedness.)
Suppose that $k \geq 5$, $0 < s_{1} < \min\left\{\frac{1}{3}, \frac{k}{4} - 1\right\}$, $\sigma > \frac{1}{2} - \frac{2}{k} + \frac{4s_{1}}{k}$, $v_{0} \in H^{\sigma}(\R)$, $0 \in I=[-T, T] \subset \R$, and $G \in Y(\R)$. Then there exists $\vartheta > 0$ such that if
\begin{equation}
\left\|\eta
(t/T)U(t)v_{0}\right\|_{X^{\sigma, \frac{1}{2}+\delta}(\SR \times \SR)} + \|G\|_{Y(I)} \leq \vartheta, \label{4.03}
\end{equation}
then the system \eqref{4.01}-\eqref{4.02} admits a unique solution $u(t, x) \in C(I; H^{\sigma}(\R))$.
\end{Theorem}
\noindent{\bf Proof.}
We define
\begin{eqnarray*}
&&B=\{v:\|v\|_{X^{\sigma,\frac{1}{2}+\delta}}\leq 2\vartheta\}
\end{eqnarray*}
and
\begin{eqnarray}
&&\Phi(v)=\eta\left(\frac{t}{T}\right)U(t)v_{0}\nonumber\\
&&\quad+\frac{\eta(t/T)}{k+1}\int_{0}^{t}U(t-s)\partial_{x}[(\eta
(s/T)v(s)+\eta
(s/T)G(s))^{k+1}]ds.\label{4.04}
\end{eqnarray}
Next, we prove that $\Phi(v)$ is a contraction mapping in $B$.
On the one hand, by using \eqref{2.040}, \eqref{4.02} and Proposition 3.1, we have
\begin{eqnarray}
&&\|\Phi(v)\|_{X^{\sigma,\frac{1}{2}+\delta}}=\left\|\eta
(t/T)U(t)v_{0}\right\|_{X^{\sigma,\frac{1}{2}+\delta}}\nonumber\\
&&\quad+
\frac{1}{k+1}\left\|\eta
(t/T)\int_{0}^{t}U(t-s)\partial_{x}[(\eta
(s/T)v(s)+\eta
(s/T)G(s))^{k+1}]ds\right\|_{X^{\sigma,\frac{1}{2}+\delta}}\nonumber\\
&&\leq \vartheta+ CT^{(k+2)\delta}\left\|\partial_{x}[(\eta
(s/T)v(s)+\eta
(s/T)G(s))^{k+1}]\right\|_{X^{\sigma,-\frac{1}{2}+(k+3)\delta}}\nonumber\\
&&\leq \vartheta+CT^{(k+2)\delta}\left(\|v\|_{X^{\sigma,b}}^{k+1}+\|G\|_{Y(I)}^{k+1}\right)\nonumber\\
&&\leq  \vartheta+C_{1}T^{(k+2)\delta}\left(\vartheta^{k+1}+\vartheta^{k+1}\right),\label{4.05}
\end{eqnarray}
where $C_{1}=2^{k+1}C$. On the other hand, by using \eqref{2.040} and Proposition 3.1,  we have
\begin{align}
&\quad\|\Phi(v_{1})-\Phi(v_{2})\|_{X^{\sigma,\frac{1}{2}+\delta}}\nonumber\\
&\leq CT^{(k+2)\delta}\Big\|\partial_{x}[(\eta
(s/T)v_{1}(s)+\eta
(s/T)G(s))^{k+1}\nonumber\\
&\qquad\qquad\qquad-(\eta
(s/T)v_{2}(s)+\eta
(s/T)G(s))^{k+1}]\Big\|_{X^{\sigma,-\frac{1}{2}+(k+3)\delta}}\nonumber\\
&\leq CT^{(k+2)\delta}\left(\|v_{1}\|_{X^{\sigma,\frac{1}{2}+\delta}}^{k}
+\|v_{2}\|_{X^{\sigma,\frac{1}{2}+\delta}}^{k}+\|G\|_{Y(I)}^{k}\right)\|v_{1}-v_{2}\|_{X^{\sigma,\frac{1}{2}+\delta}}\nonumber\\
&\leq C_{1}T^{(k+2)\delta}(\vartheta^{k}+\vartheta^{k})\|v_{1}-v_{2}\|_{X^{\sigma,\frac{1}{2}+\delta}},\label{4.06}
\end{align}
where $C_{1}=2^{k+1}C$. We take a sufficiently small $\vartheta=\min\left\{(4C_{1})^{-\frac{1}{k}}, (4C_{1}T^{(k+2)\delta})^{-\frac{1}{k}}\right\}$, such that
\begin{eqnarray}
&&C_{1}T^{(k+2)\delta}\left(\vartheta^{k+1}+\vartheta^{k+1}\right)\leq \vartheta,\quad C_{1}T^{(k+2)\delta}(\vartheta^{k}+\vartheta^{k})\leq \frac{1}{2}.\label{4.07}
\end{eqnarray}

From \eqref{4.04}-\eqref{4.07}, we have that $\Phi(v)$ is a contraction mapping in $B$.

The proof of Theorem 4.1 is completed.

\bigskip

\noindent {\bf Proof of Theorem  1.1}

\bigskip

It follows from Lemma 3.1 and \cite[Lemma 3.3]{DLM2019} that for almost every $\omega\in\Omega$,  there exists a time interval $I_{\omega}$ containing $0$, such that
\begin{eqnarray}
&&\|U(t)f^{\omega}\|_{Y(I_{\omega})}\leq \vartheta.\label{4.08}
\end{eqnarray}
We denote $z(t,x)=U(t)f^{\omega}(x)$ and $v(t,x)=u(t,x)-z(t,x)$, where $u(t,x)$ is a solution of \eqref{1.01} with random data $u(0,x)=f^{\omega}(x)$. Then, by using \eqref{1.01}, we have
\begin{eqnarray}
&&v_{t}+v_{xxx}+\frac{1}{k+1}\partial_{x}[(v+z)^{k+1}]=0,\label{4.09}\\
&&v(0,x)=v_{0}(x)=0.\label{4.010}
\end{eqnarray}
From \eqref{4.08}-\eqref{4.010}, we have
\begin{eqnarray}
&&\|\eta
(t/T)U(t)v_{0}\|_{X^{\sigma,\frac{1}{2}+\delta}}+\|z\|_{Y(I_{\omega})}=\|z\|_{Y(I_{\omega})}\leq \vartheta.\label{4.011}
\end{eqnarray}
By using \eqref{4.011} and Theorem 4.1, we have that Theorem 1.1 is valid.

The proof of Theorem 1.1 is completed.

\bigskip

\bigskip

\section{Proof of Theorem  1.2}

\setcounter{equation}{0}

 \setcounter{Theorem}{0}

\setcounter{Lemma}{0}

\setcounter{Proposition}{0}

\setcounter{section}{5}

By using \eqref{4.05}-\eqref{4.07}, we choose $\sigma=\frac{1}{2}+\epsilon$. Then, for  almost every $\omega\in\Omega$, we have
\begin{eqnarray}
&&\|u(t,x)-U(t)f^{\omega}(x)\|_{X^{\frac{1}{2}+\epsilon,\frac{1}{2}+\delta}}
\leq C.\label{5.01}
\end{eqnarray}
We define
\begin{eqnarray*}
&&\tilde{\Omega}=\left\{\omega:\lim_{t\rightarrow0}\|u(t,x)-U(t)f^{\omega}(x)\|_{L_{x}^{\infty}}=0\right\}.
\end{eqnarray*}
By using  $X^{\frac{1}{2}+\epsilon,\frac{1}{2}+\delta}(I_{\omega}\times\R)\hookrightarrow C(I_{\omega}; H^{\frac{1}{2}+\epsilon}(\R))$ and the Sobolev embedding theorem $H^{\frac{1}{2}+\epsilon}(\R)\hookrightarrow L^{\infty}(\R)$, we obtain
\begin{align}
\|u(t,x)-U(t)f^{\omega}(x)\|_{L^{\infty}}&\leq C\|u(t,x)-U(t)f^{\omega}(x)\|_{H^{\frac{1}{2}+\epsilon}}\nonumber\\
&\leq C\|u(t,x)-U(t)f^{\omega}(x)\|_{X^{\frac{1}{2}+\epsilon,\frac{1}{2}+\delta}}\leq C.\label{5.02}
\end{align}
From \eqref{5.02}, for almost every $\omega\in\Omega$, we have
\begin{eqnarray}
&&\lim_{t\rightarrow0}\|u(t,x)-U(t)f^{\omega}(x)\|_{L_{x}^{\infty}}=0.\label{5.03}
\end{eqnarray}
By using \eqref{5.03}, we know that $\omega\in\tilde{\Omega}$ and
\begin{eqnarray}
&&1\leq \mathbb{P}\left(\tilde{\Omega}\right)\leq 1.\label{5.04}
\end{eqnarray}
It follows from \eqref{5.04} that
\begin{eqnarray}
&&\mathbb{P}\left(\left\{\omega:\lim_{t\rightarrow0}\left\|u(t,x)-U(t)f^{\omega}(x)\right\|_{L_{x}^{\infty}}=0\right\}\right)=1.\label{5.05}
\end{eqnarray}

The proof of Theorem 1.2 is completed.

\bigskip

\bigskip

\section{Proof of Theorem  1.3}

\setcounter{equation}{0}

 \setcounter{Theorem}{0}

\setcounter{Lemma}{0}

\setcounter{Proposition}{0}

\setcounter{section}{6}
Noting that
\begin{eqnarray*}
&&\|u\|_{X_{T}^{s}}=\left\|\|\Delta_{j}u\|_{Y_{T}^{s}}\right\|_{\ell^{2}},\\ &&\|u\|_{Y_{T}^{s}}=\left\|u\right\|_{L_{T}^{\infty}H^{s}}+\left\|\langle D\rangle^{s+1-2\epsilon}u\right\|_{L_{x}^{\infty}L_{T}^{\frac{2}{1-\epsilon_{1}}}}
+\left\|\langle D\rangle^{s-\frac{k-2}{2k}-\epsilon_{2}}u\right\|_{L_{x}^{k}L_{T}^{\frac{k}{\epsilon_{1}}}}.
\end{eqnarray*}
By using \eqref{2.013}, \eqref{2.015} and \eqref{2.016}, we have
\begin{align}
&\quad\left\|\int_{0}^{t}U(t-\tau)\partial_{x}\left(\prod_{j=1}^{k+1}u_{j}\right)d\tau\right\|_{X_{T}^{\mu+s}}\nonumber\\
&\leq
C(T)\left(\sum_{j\geq0}2^{2(\mu+s+2\epsilon)j}\left\|\Delta_{j}\left(\prod_{j=1}^{k+1}u_{j}\right)\right\|
_{L_{x}^{1}L_{T}^{\frac{2}{1+\epsilon_{1}}}}^{2}\right)^{\frac{1}{2}},\label{6.01}
\end{align}
where $C(T)=C\max\left\{T^{\frac{4\epsilon^{2}}{3+2\epsilon}},\, T^{\frac{8\epsilon^{2}}{3+2\epsilon}},\, T^{\frac{4\epsilon^{2}}{3+2\epsilon}+\frac{\epsilon_1}{k}} \right\}$.
We decompose $\prod\limits_{j=1}^{k+1}u_{j}$ as follows
\begin{eqnarray*}
&&\prod_{j=1}^{k+1}u_{j}=\sum_{j_{1}\geq0}\Delta_{j_{1}}u_{1}\sum_{j_{2}\geq0}\Delta_{j_{2}}u_{2}\ldots\sum_{j_{k+1}\geq0}\Delta_{j_{k+1}}u_{k+1}
=\sum_{m=1}^{(k+1)!}H_{m},
\end{eqnarray*}
where
\begin{eqnarray*}
&&H_{1}=\sum_{j_{1}\geq j_{2}\geq\ldots\geq j_{k+1}}\Delta_{j_{1}}u_{1}\Delta_{j_{2}}u_{2}\ldots \Delta_{j_{k+1}}u_{k+1}.
\end{eqnarray*}
$H_{j}(2\leq j\leq (k+1)!)$ have similar forms and can be estimated in a process similar to that of $H_{1}$, so we treat only $H_{1}$.
Observing that
\begin{align*}
H_{1}&=\sum_{j_{1}\geq j_{2}\geq\ldots\geq j_{k+1}}\Delta_{j_{1}}u_{1}\Delta_{j_{2}}u_{2}\ldots \Delta_{j_{k+1}}u_{k+1}\nonumber\\
&=\Delta_{0}u_{1}\Delta_{0}u_{2}\ldots \Delta_{0}u_{k+1}+\sum_{j_{1}\geq1}\Delta_{j_{1}}u_{1}\Delta_{0}u_{2}\ldots \Delta_{0}u_{k+1}\nonumber\\
&\quad+\sum_{j_{1}\geq j_{2}\geq1}\Delta_{j_{1}}u_{1}\Delta_{j_{2}}u_{2}\ldots \Delta_{0}u_{k+1}+\ldots\nonumber\\
&\quad+\sum_{j_{1}\geq j_{2}\geq\ldots\geq j_{k+1}\geq1}\Delta_{j_{1}}u_{1}\Delta_{j_{2}}u_{2}\ldots \Delta_{j_{k+1}}u_{k+1}\nonumber\\
&=\sum_{n=1}^{k+1}H_{1n}.
\end{align*}
Since other cases can be obtained by similar proofs, we only consider $H_{11},\, H_{12},\, H_{13},\, H_{1(k+1)}$.

\noindent For $H_{11}$, by using H\"{o}lder inequality, Lemma 2.9 and $\mu+s\geq0$, we have
\begin{align}
&\quad\left(\sum_{j\leq k+1}2^{2(\mu+s+2\epsilon)j}\|\Delta_{j}H_{11}\|_{L_{x}^{1}L_{T}^{\frac{2}{1+\epsilon_{1}}}}^{2}\right)^{\frac{1}{2}}\nonumber\\
&\leq
C\left(\sum_{j\leq k+1}2^{2(\mu+s+2\epsilon)j}\|\Delta_{0}u_{1}\Delta_{0}u_{2}\ldots \Delta_{0}u_{k+1}\|_{L_{x}^{1}L_{T}^{\frac{2}{1+\epsilon_{1}}}}^{2}\right)^{\frac{1}{2}}\nonumber\\
&\leq C\left(\sum_{j\leq k+1}2^{2(\mu+s+2\epsilon)j}\|\Delta_{0}u_{1}\|_{L_{x}^{\infty}L_{T}^{\frac{2}{1-\epsilon_{1}}}}^{2}
\prod_{n=2}^{k+1}\|\Delta_{0}u_{n}\|_{L_{x}^{k}L_{T}^{\frac{k}{\epsilon_{1}}}}^{2}\right)^{\frac{1}{2}}\nonumber\\
&\leq C\left(\sum_{j\leq k+1}2^{2(\mu+s+2\epsilon)j}\right)^{\frac{1}{2}} \|\Delta_{0}u_{1}\|_{L_{x}^{\infty}L_{T}^{\frac{2}{1-\epsilon_{1}}}}\prod_{n=2}^{k+1}\|\Delta_{0}u_{n}\|_{L_{x}^{k}L_{T}^{\frac{k}{\epsilon_{1}}}}\nonumber\\
&\leq C\prod_{j=1}^{k+1}\|u_{j}\|_{X_{T}^{s}}.\label{6.02}
\end{align}
For $H_{12}$, by using H\"{o}lder inequality, Lemma 2.9 and $\mu\leq 1-6\epsilon$, we have
\begin{align}
&\quad\left(\sum_{j\geq 1}2^{2(\mu+s+2\epsilon)j}\|\Delta_{j}H_{12}\|_{L_{x}^{1}L_{T}^{\frac{2}{1+\epsilon_{1}}}}^{2}\right)^{\frac{1}{2}}\nonumber\\
&\leq
C\left(\sum_{j\geq 1}2^{-\epsilon j}\left[\sum_{j_{1}\geq C+j}2^{(\mu+s+3\epsilon)j_{1}}\|\Delta_{j_{1}}u_{1}\Delta_{0}u_{2}\ldots \Delta_{0}u_{k+1}\|_{L_{x}^{1}L_{T}^{\frac{2}{1+\epsilon_{1}}}}\right]^{2}\right)^{\frac{1}{2}}\nonumber\\
&\leq C\sum_{j_{1}\geq 1}2^{(1+s-3\epsilon)j_{1}}\|\Delta_{j_{1}}u_{1}\|_{L_{x}^{\infty}L_{T}^{\frac{2}{1-\epsilon_{1}}}}
\prod_{n=2}^{k+1}\|\Delta_{0}u_{n}\|_{L_{x}^{k}L_{T}^{\frac{k}{\epsilon_{1}}}}\nonumber\\
&\leq C\left(\sum_{j_{1}\geq 1}2^{2(1+s-2\epsilon)j_{1}}\|\Delta_{j_{1}}u_{1}\|_{L_{x}^{\infty}L_{T}^{\frac{2}{1-\epsilon_{1}}}}^{2}\right)^{\frac{1}{2}}
\prod_{n=2}^{k+1}\|\Delta_{0}u_{n}\|_{L_{x}^{k}L_{T}^{\frac{k}{\epsilon_{1}}}}\nonumber\\
&\leq C\prod_{j=1}^{k+1}\|u_{j}\|_{X_{T}^{s}}.\label{6.03}
\end{align}
For $H_{13}$, by using H\"{o}lder inequality, Lemma 2.9 and $s\geq\frac{\mu-2}{k}+\frac{1}{2}+\frac{6\epsilon}{k}+\epsilon_{2}$,\, $0\leq \mu\leq \min\{1-6\epsilon,\,1-6\epsilon+k(s-(\frac{1}{2}-\frac{1}{k})-\epsilon_{2})\}$, we have
\begin{align}
&\quad\left(\sum_{j\geq 1}2^{2(\mu+s+2\epsilon)j}\|\Delta_{j}H_{12}\|_{L_{x}^{1}L_{T}^{\frac{2}{1+\epsilon_{1}}}}^{2}\right)^{\frac{1}{2}}\nonumber\\
&\leq
C\left(\sum_{j\geq 1}2^{-\epsilon j}\left[\sum_{j_{1}\geq j_{2}\geq1}2^{(\mu+s+3\epsilon)j_{1}}\|\Delta_{j_{1}}u_{1}\Delta_{j_{2}}u_{2}\ldots \Delta_{0}u_{k+1}\|_{L_{x}^{1}L_{T}^{\frac{2}{1+\epsilon_{1}}}}\right]^{2}\right)^{\frac{1}{2}}\nonumber\\
&\leq C\left(\sum_{j\geq 1}2^{-\epsilon j}\left[\sum_{j_{1}\geq j_{2}\geq 1}2^{(s+1-3\epsilon)j_{1}}2^{(s-\frac{k-2}{2k}-\epsilon_{2})j_{2}}\|\Delta_{j_{1}}u_{1}\|_{L_{x}^{\infty}L_{T}^{\frac{2}{1-\epsilon_{1}}}}
\|\Delta_{j_{2}}u_{2}\|_{L_{x}^{k}L_{T}^{\frac{k}{\epsilon_{1}}}}\right]^{2}\right)^{\frac{1}{2}}\nonumber\\
&\quad\times\prod_{n=3}^{k+1}\|\Delta_{0}u_{n}\|_{L_{x}^{k}L_{T}^{\frac{k}{\epsilon_{1}}}}\nonumber\\
&\leq C\left(\sum_{j_{1}\geq 1}2^{2(1+s-2\epsilon)j_{1}}\|\Delta_{j_{1}}u_{1}\|_{L_{x}^{\infty}L_{T}^{\frac{2}{1-\epsilon_{1}}}}^{2}\right)^{\frac{1}{2}}
\left(\sum_{j_{2}\geq 1}2^{2(s-\frac{k-2}{2k}-\epsilon_{2})j_{2}}\|\Delta_{j_{2}}u_{2}\|_{L_{x}^{k}L_{T}^{\frac{k}{\epsilon_{1}}}}^{2}\right)^{\frac{1}{2}}\nonumber\\
&\quad\times\prod_{n=3}^{k+1}\|\Delta_{0}u_{n}\|_{L_{x}^{k}L_{T}^{\frac{k}{\epsilon_{1}}}}
\leq C\prod_{j=1}^{k+1}\|u_{j}\|_{X_{T}^{s}}.\label{6.04}
\end{align}
For $H_{1(k+1)}$, by using H\"{o}lder inequality, Lemma 2.9 and $s\geq\frac{\mu-2}{k}+\frac{1}{2}+\frac{6\epsilon}{k}+\epsilon_{2}$,\, $0\leq \mu\leq \min\{1-6\epsilon,\,1-6\epsilon+k(s-(\frac{1}{2}-\frac{1}{k})-\epsilon_{2})\}$, we have
\begin{align}
&\quad\left(\sum_{j\geq 1}2^{2(\mu+s+2\epsilon)j}\|\Delta_{j}H_{1(k+1)}\|_{L_{x}^{1}L_{T}^{\frac{2}{1+\epsilon_{1}}}}^{2}\right)^{\frac{1}{2}}\nonumber\\
&\leq
C\left(\sum_{j\geq 1}2^{-\epsilon j}\left[\sum_{j_{1}\geq j_{2}\geq\ldots\geq j_{k+1}\geq1}2^{(\mu+s+3\epsilon)j_{1}}\|\Delta_{j_{1}}u_{1}\Delta_{j_{2}}u_{2}\ldots \Delta_{j_{k+1}}u_{k+1}\|_{L_{x}^{1}L_{T}^{\frac{2}{1+\epsilon_{1}}}}\right]^{2}\right)^{\frac{1}{2}}\nonumber\\
&\leq C\sum_{j_{1}\geq j_{2}\geq\ldots\geq j_{k+1}\geq1}2^{(s+1-3\epsilon)j_{1}}2^{\sum_{n=2}^{k+1}(s-\frac{k-2}{2k}-\epsilon_{2})j_{n}}\|\Delta_{j_{1}}u_{1}\|_{L_{x}^{\infty}L_{T}^{\frac{2}{1-\epsilon_{1}}}}
\prod_{n=2}^{k+1}\|\Delta_{j_{n}}u_{n}\|_{L_{x}^{k}L_{T}^{\frac{k}{\epsilon_{1}}}}\nonumber\\
&\leq C\left(\sum_{j_{1}\geq 1}2^{2(1+s-2\epsilon)j_{1}}\|\Delta_{j_{1}}u_{1}\|_{L_{x}^{\infty}L_{T}^{\frac{2}{1-\epsilon_{1}}}}^{2}\right)^{\frac{1}{2}}
\prod_{n=2}^{k+1}\left(\sum_{j_{n}\geq 1}2^{2(s-\frac{k-2}{2k}-\epsilon_{2})j_{n}}\|\Delta_{j_{n}}u_{n}\|_{L_{x}^{k}L_{T}^{\frac{k}{\epsilon_{1}}}}^{2}\right)^{\frac{1}{2}}
\nonumber\\
&\leq C\prod_{j=1}^{k+1}\|u_{j}\|_{X_{T}^{s}}.\label{6.05}
\end{align}
From \eqref{6.01}-\eqref{6.05}, we have
\begin{align}
&\quad\left\|\int_{0}^{t}U(t-\tau)\partial_{x}\left(\prod_{j=1}^{k+1}u_{j}\right)d\tau\right\|_{X_{T}^{\mu+s}}\nonumber\\
&\leq
C(T)\left(\sum_{j\geq0}2^{2(\mu+s+2\epsilon)j}\left\|\Delta_{j}\left(\prod_{j=1}^{k+1}u_{j}\right)\right\|
_{L_{x}^{1}L_{T}^{\frac{2}{1+\epsilon_{1}}}}^{2}\right)^{\frac{1}{2}}\nonumber\\
&\leq C(T)\prod_{j=1}^{k+1}\|u_{j}\|_{X_{T}^{s}},\label{6.06}
\end{align}
where $C(T)=C\max\left\{T^{\frac{4\epsilon^{2}}{3+2\epsilon}},\, T^{\frac{8\epsilon^{2}}{3+2\epsilon}},\, T^{\frac{4\epsilon^{2}}{3+2\epsilon}+\frac{\epsilon_1}{k}} \right\}$.

The proof of Theorem 1.3 is completed.

\bigskip

\bigskip

\section{Proof of Theorem  1.4}

\setcounter{equation}{0}

 \setcounter{Theorem}{0}

\setcounter{Lemma}{0}

\setcounter{Proposition}{0}

\setcounter{section}{7}
Before proving Theorem 1.4, we first state a local well-posedness result, which is crucial for the  proof of Theorem 1.4.

\begin{Theorem}\label{Theorem 7.1}(Local well-posedness with rough data.)
Suppose that $k\geq4$, $\epsilon, \epsilon_{2}>0$, $f(x)\in H^{s}(\R)$ and $s\geq\frac{1}{2}-\frac{2}{k}+\frac{6\epsilon}{k}+\epsilon_{2}$.
Then, there exists $T>0$ such that  \eqref{1.01} admits a unique solution $u(t,x)\in C([0,T];H^{s}(\R))$ with rough data $f(x)$.
\end{Theorem}

\noindent{\bf Proof.}
Theorem 7.1 has already been proved in \cite{KPV1993}. Here, we prove it by applying Theorem 1.3.
By Duhamel's principle and \eqref{1.01} with the rough data $f(x)$, we know that
\begin{eqnarray}
&&u(t)=U(t)f+\frac{1}{k+1}\int_{0}^{t}U(t-s)\partial_{x}(u^{k+1})ds.\label{7.01}
\end{eqnarray}
We define $B_{1}=\{u:\|u\|_{X_{T}^{s}}\leq 2R=2C\|f\|_{H^{s}}\}$ and
\begin{eqnarray}
&&\Phi_{1}(u)=U(t)f+\frac{1}{k+1}\int_{0}^{t}U(t-s)\partial_{x}(u^{k+1})ds.\label{7.02}
\end{eqnarray}
On the one hand, by the definition of $X_{T}^{s}$ and \eqref{2.07} and \eqref{2.019}, we obtain
\begin{eqnarray}
&&\|U(t)f\|_{X_{T}^{s}}\leq C\|f\|_{H^{s}}\leq R.\label{7.03}
\end{eqnarray}
On the other hand, by using Theorem 1.3 with $\mu=0$, we have
\begin{eqnarray}
&&\frac{1}{k+1}\left\|\int_{0}^{t}U(t-s)\partial_{x}(u^{k+1})ds\right\|_{X_{T}^{s}}\leq CT^{\frac{4\epsilon^{2}}{3+2\epsilon}}\|u\|_{X_{T}^{s}}^{k+1}\leq CT^{\frac{4\epsilon^{2}}{3+2\epsilon}}(2R)^{k+1}.\label{7.04}
\end{eqnarray}
Next, we prove that $\Phi_{1}(u)$ is a contraction mapping in $B_{1}$. By using \eqref{7.03} and \eqref{7.04}, we have
\begin{align}
\|\Phi_{1}(u)\|_{X_{T}^{s}}&\leq \|U(t)f\|_{X_{T}^{s}}+\frac{1}{k+1}\left\|\int_{0}^{t}U(t-s)\partial_{x}(u^{k+1})ds\right\|_{X_{T}^{s}}\nonumber\\
&\leq C\|f\|_{H^{s}}+CT^{\frac{4\epsilon^{2}}{3+2\epsilon}}\|u\|_{X_{T}^{s}}^{k+1}\nonumber\\
&\leq (R+CT^{\frac{4\epsilon^{2}}{3+2\epsilon}}(2R)^{k+1}).\label{7.05}
\end{align}
Moreover, for $u,v\in X_{T}^{s}$, we have
\begin{align}
\|\Phi_{1}(u)-\Phi_{1}(v)\|_{X_{T}^{s}}&\leq \frac{1}{k+1}\left\|\int_{0}^{t}U(t-s)\partial_{x}(u^{k+1}-v^{k+1})ds\right\|_{X_{T}^{s}}\nonumber\\
&\leq CT^{\frac{4\epsilon^{2}}{3+2\epsilon}}(\|u\|_{X_{T}^{s}}^{k}+\|v\|_{X_{T}^{s}}^{k})\|u-v\|_{X_{T}^{s}}\nonumber\\
&\leq 2CT^{\frac{4\epsilon^{2}}{3+2\epsilon}}(2R)^{k}\|u-v\|_{X_{T}^{s}}.\label{7.06}
\end{align}
We take a sufficiently small $T>0$, such that
\begin{eqnarray}
&&CT^{\frac{4\epsilon^{2}}{3+2\epsilon}}(2R)^{k}\leq 1,\, 2CT^{\frac{4\epsilon^{2}}{3+2\epsilon}}(2R)^{k}\leq \frac{1}{4}.\label{7.07}
\end{eqnarray}
From \eqref{7.05}-\eqref{7.07}, we have that $\Phi_{1}(u)$ is a contraction mapping in $B_{1}$.

The proof of Theorem 7.1 is completed.

\bigskip
\noindent{\bf Proof of Theorem  1.4.}
\bigskip

Since $f\in\hat{L}^{\infty}(\R)$, we have
\begin{align}
\int_{\SR}|\mathscr{F}_{x}(U(t)f)(\xi)|d\xi&=\int_{\SR}|e^{it\xi^{3}}\mathscr{F}_{x}f(\xi)|d\xi\nonumber\\
&=\int_{\SR}|\mathscr{F}_{x}f(\xi)|d\xi\leq C.\label{7.08}
\end{align}
From \eqref{7.08}, we obtain $U(t)f(x)\in \hat{L}^{\infty}$, which implies that
\begin{eqnarray}
&&\lim_{|x|\rightarrow\infty}U(t)f=0.\label{7.09}
\end{eqnarray}
By using Theorem 7.1 and Theorem 1.3 with $\mu=ks-\frac{k}{2}+2-6\epsilon-k\epsilon_{2}$, when $s\geq\frac{1}{2}-\frac{2}{k+1}+\frac{7\epsilon}{k+1}+\frac{k\epsilon_{2}}{k+1}$, $T>0$ sufficiently small, we have
\begin{eqnarray}
&&\left\|\int_{0}^{t}U(t-s)\partial_{x}[u^{k+1}(s,x)]ds\right\|_{X_{T}^{\frac{1}{2}+\epsilon}}\leq C\|u\|_{X_{T}^{s}}^{k+1}\leq C.\label{7.010}
\end{eqnarray}
From \eqref{7.010}, for $t\in[0,T]$, we have
\begin{align}
\left\|\int_{0}^{t}U(t-s)\partial_{x}[u^{k+1}(s,x)]ds\right\|_{C([0,T]; H^{\frac{1}{2}+\epsilon})}&\leq
\left\|\int_{0}^{t}U(t-s)\partial_{x}[u^{k+1}(s,x)]ds\right\|_{X_{T}^{\frac{1}{2}+\epsilon}}\nonumber\\
&\leq C.\label{7.011}
\end{align}
It follows from \eqref{7.011} that
\begin{eqnarray}
&&\lim_{|x|\rightarrow\infty}(u(t,x)-U(t)f)=\lim_{|x|\rightarrow\infty}\int_{0}^{t}U(t-s)\partial_{x}[u^{k+1}(s,x)]ds=0.\label{7.012}
\end{eqnarray}
By using \eqref{7.09} and \eqref{7.012}, for $t\in[0,T]$,  we have
\begin{eqnarray}
&&\lim_{|x|\rightarrow\infty}u(t,x)=0.\label{7.013}
\end{eqnarray}

The proof of Theorem 1.4 is completed.

\bigskip

\bigskip

\section{Proof of Theorem  1.5}

\setcounter{equation}{0}

 \setcounter{Theorem}{0}

\setcounter{Lemma}{0}

\setcounter{Proposition}{0}

\setcounter{section}{8}

By using \eqref{4.05}-\eqref{4.07}, for  almost every $\omega\in\Omega$ and $t\in I_{\omega}$, we have
\begin{eqnarray}
&&\left\|\int_{0}^{t}U(t-s)\partial_{x}[(v(s,x)+z(s,x))^{k+1}]ds\right\|_{H^{\frac{1}{2}+\epsilon}}\leq C.\label{8.01}
\end{eqnarray}
We define
\begin{eqnarray}
&&\tilde{\Omega}_{1}=\left\{\omega: \forall t\in I_{\omega}, \lim_{|x|\rightarrow\infty}\left(u(t,x)-U(t)f^{\omega}(x)\right)=0\right\}.\label{8.02}
\end{eqnarray}
By using \eqref{8.01}, for  almost every $\omega\in\Omega$, we  have
\begin{eqnarray}
&&\lim_{|x|\rightarrow\infty}\int_{0}^{t}U(t-s)\partial_{x}[(v(s,x)+z(s,x))^{k+1}]ds=0.\label{8.03}
\end{eqnarray}
From \eqref{8.03}, we have
\begin{eqnarray}
&&\lim_{|x|\rightarrow\infty}\left(u(t,x)-U(t)f^{\omega}(x)\right)=0.\label{8.04}
\end{eqnarray}
It follows from \eqref{8.04} that $\omega\in\tilde{\Omega}_{1}$ and
\begin{eqnarray}
&&1\leq \mathbb{P}\left(\tilde{\Omega}_{1}\right)\leq 1.\label{8.05}
\end{eqnarray}
By using \eqref{8.05}, we have
\begin{eqnarray}
&&\mathbb{P}\left(\left\{\omega: \forall t\in I_{\omega}, \lim_{|x|\rightarrow\infty}\left(u(t,x)-U(t)f^{\omega}(x)\right)=0\right\}\right)=1.\label{8.06}
\end{eqnarray}

The proof of Theorem 1.5 is completed.

\bigskip

\bigskip

\bigskip
\bigskip

\leftline{\large \bf Acknowledgments}
The last author is supported by the National Natural Science Foundation of
 China under the grant number  of  12371245.

\end{document}